\documentclass{amsproc}
\usepackage{amssymb}
\usepackage{dsfont}
\usepackage[matrix, arrow, curve]{xy}
\usepackage{enumerate}
\usepackage{stmaryrd}
\newtheorem{theorem}{Theorem}[section]
\newtheorem{lemma}[theorem]{Lemma}
\newtheorem{proposition}[theorem]{Proposition}
\newtheorem{cor}[theorem]{Corollary}  

\newtheorem{definition}[theorem]{Definition}
\usepackage[pdftex]{graphicx,color}
\theoremstyle{definition}
\newtheorem{rem}[theorem]{Remark}

\newcommand{\argch}{\mathrm{Argch}}

\usepackage{tikz}
\usetikzlibrary{patterns}
\tikzstyle{vertex}=[
    circle,
    minimum size = 0.1cm,
    draw=black,
    thick,
    fill=black
]

\usetikzlibrary{shapes.geometric, arrows}
\usepackage{url}
\usepackage{hyperref}
\usepackage[a4paper,]{geometry}
\usepackage{pdfsync}

\numberwithin{equation}{section}

\makeatletter
\@namedef{subjclassname@2024}{\textup{2024} Mathematics Subject Classification}
\makeatother

\begin{document}

\title
[]
{A discrete-time Matsumoto--Yor theorem}

\author{}
\address{}

\subjclass[2024]{Primary 60J10, 60J65; Secondary 60B15}
\keywords{Brownian motion, Dufresne identity, Generalized Inverse Gaussian distributions, \\
Intertwining relation, Matsumoto--Yor theorem, Modified Bessel--Macdonald functions}

\author{Charlie H\'erent}
\email  {charlie.herent@ens-rennes.fr}
\address{Universit\'e Paris Cit\'e, CNRS, MAP5, F-75006 Paris, France \\
and Universit\'e Paris-Est, CNRS,
Institut Gaspard Monge, Champs-Sur-Marne, France}

\maketitle



\begin{abstract}
We study a random walk on the subgroup of lower triangular matrices of $SL_2$, with i.i.d. increments. We prove that the process of the lower corner of the random walk satisfies a Rogers--Pitman criterion to be a Markov chain if and only if the increments are distributed according to a Generalized Inverse Gaussian (GIG) law on their diagonals. For this, we prove a new characterization of these laws. We prove a discrete-time version of the Dufresne identity. \\
We show how to recover the Matsumoto–-Yor theorem by taking the continuous limit of the random walk.
\end{abstract}

\maketitle

\tableofcontents

\section{Introduction}

\subsection{Background and literature}

Let $( B_t ^{(\mu)} :  t \geq 0)$ be a one-dimensional Brownian motion with drift $\mu \in \mathbb{R}$.
Matsumoto--Yor's theorem (th. 1.6 from \cite{MY1}) states that the continuous process $( \mathcal{Z}_t ^{(\mu)} :  t \geq 0)$ defined by
\begin{equation} \label{MY}
\mathcal{Z}_t ^{(\mu)} := e^{B_t ^{(\mu)}} \int_{0}^t e^{- 2 B_s ^{(\mu)}} ds 
\end{equation}
is a diffusion process on $\mathbb{R}_+$ with an infinitesimal generator given by

\begin{equation} \label{MYGenerateur}
\frac{1}{2} z^2 \dfrac{d^2}{dz^2} + \left[ \left( \frac{1}{2} + \mu \right) z + \left( \frac{K_{1- \mu}}{K_{\mu}} \right) \left( \frac{1}{z} \right) \right] \frac{d}{dz}
\end{equation}
where $K_{\mu}$ is a modified Bessel function of the second kind, also called Macdonald function.
This result is a geometric version of Pitman's theorem \cite{Pitman} of 1975 which states that the stochastic process 
$$( B^{(0)} _t - 2 \inf_{0 \leq s \leq t} B^{(0)} _s :  t \geq 0 )$$ is distributed as the three-dimensional Bessel process. Indeed, from Matsumoto--Yor's theorem, with a Laplace approximation argument and using the scaling property of Brownian motion, we recover Pitman's theorem as it is done in \cite{MY1}. \\

Many generalizations of Pitman's theorem have appeared in last decades. In particular, Biane, Bougerol and O'Connell in \cite{BBO} and \cite{BBO2} have extended Pitman's theorem in the context of semisimple Lie algebra and finite Coxeter groups. More recently, \cite{Bou-Def} and \cite{Def-Her} have established Pitman type theorems in the context of affine Lie algebra. \\

Here, we focus on the interpretation of Matsumoto--Yor's theorem from the work of Chhaibi \cite{Reda}, \cite{Reda1}, \cite{Reda2}, \cite{Reda3}. Chhaibi extends Matsumoto--Yor's theorem following the work of O'Connell \cite{OconnellFirst} in the framework of random matrices. In the case of $2 \times 2$ matrices this extension reduces to the study of the following SDE in the Stratonovitch sense (example 2.4 in \cite{Reda2}): 
\begin{equation}\label{EDSReda}
d\mathbf{B}_t = \mathbf{B}_t  \circ \begin{pmatrix}
dB_t ^{(\mu)} &0 \\
dt & - dB_t ^{(\mu)}
\end{pmatrix}
\ \text{with} \  \mathbf{B}_0 = I_2
\end{equation}
whose solution is given by
\begin{equation}\label{SolutionEDSREDA}
\mathbf{B}_t = \begin{pmatrix}
e^{B_t ^{(\mu)}} &0 \\
e^{B_t ^{(\mu)}} \int_{0}^t e^{- 2 B_s ^{(\mu)}} ds  & e^{-B_t ^{(\mu)}}
\end{pmatrix}
\end{equation}
Matsumoto--Yor's process \eqref{MY} appearing in the lower corner of the matrix, is called the highest weight process for reasons explained in \cite{Reda}. Chhaibi proved in \cite{Reda2} (th. 3.1) that the highest weight process is a diffusion with an explicit infinitesimal generator. The main object of this paper is the study of a discrete-time version of this diffusion. \\

The following result, often called \textit{Dufresne identity}, states the equality in law for $\mu >0$:
\begin{equation}\label{DufresneIdentity}
\int_{0}^{+ \infty} e^{- 2 B_s ^{(\mu)}} ds \overset{\text{law}}{=} \frac{1}{2 \xi}
\end{equation}
where $\xi$ is a random variable with $Gamma(\mu)$ distribution. 
This identity obtained in \cite{Dufresne} plays also an important role in the work of Matsumoto--Yor in \cite{MY1} and \cite{MY2}, and has many applications in finance and physics. See also in \cite{AristaOconnell} and \cite{RiderValko} some generalizations of this identity. We will prove a discrete-time version of this identity in our study. 

\subsection{Contributions of this work} In this article, we consider a discrete-time version of the SDE \eqref{EDSReda}. 
Let $(\gamma_n)_{n \in \mathbb{N}}$ be a family of identically distributed independent random variables with $\mathcal{C}^1$ density function and $(X_n)_{n \in \mathbb{N}^*}$, $(Z_n)_{n \in \mathbb{N}^*}$ be two discrete-time processes defined by the random walk $(b_{n} ^{(\delta)})_{n \in \mathbb{N}}$, with $\delta \in \mathbb{R}^*$ a deterministic parameter:

\begin{equation} \label{model0}
b_{n+1} ^{(\delta)} := b_{n} ^{(\delta)} g_n ^{(\delta)}  \ \text{with} \ b_{0} ^{(\delta)} = I_2, \ \text{hence} \ b_{n} ^{(\delta)} = g_{n-1} ^{(\delta)} \cdots  g_{0} ^{(\delta)}
\end{equation}
where,
$$b_{n} ^{(\delta)} := \begin{pmatrix}
X_{n} &0 \\
Z_{n} & X_{n} ^{-1}
\end{pmatrix} \ \text{and} \ g_n ^{(\delta)} := \begin{pmatrix}
\gamma_{n} & 0 \\
\delta & \gamma_{n}  ^{-1}
\end{pmatrix}.$$

The process $(Z_n)_{n \in \mathbb{N}^*}$ is the analog of highest weight process and is a function of the Markov process $(b_{n} ^{(\delta)})_{n \in \mathbb{N}}$. Rogers--Pitman \cite{RogersPitman} gave a criterion for a function of a Markov process to be a Markov process. In the paper we give a necessary and sufficient
condition on the increments to guarantee that this criterion is satisfied. More precisely we prove that $(Z_n)_{n \in \mathbb{N}^*}$ is a Markov chain and that there exists an intertwining relation between $(Z_n)_{n \in \mathbb{N}^*}$ and $(X_n)_{n \in \mathbb{N}^*}$ if and only if $\gamma_i$ is distributed according to a Generalized Inverse Gaussian (GIG) law. We start by proving the sufficient condition, then we
establish a new characterization of GIG distributions and deduce the necessary condition. GIG laws already appeared in the work of Matsumoto--Yor who showed that the process \eqref{MY} is intertwined with $( B_t ^{(\mu)} : t \geq 0)$ and the intertwining kernel may be expressed in terms of GIG laws as it is done in \cite{MY2}. Chhaibi also obtained a generalization of this intertwining relation in \cite{Reda}. \\

Recall that every lower triangular matrix of $SL_2$ can be uniquely factorized in the form $NA$ or $AN$ with $N$ a unipotent matrix and $A$ a diagonal matrix,
$$\begin{pmatrix}
x &0 \\
z & x ^{-1}
\end{pmatrix} = 
\begin{pmatrix}
1 &0 \\
x^{-1} z& 1
\end{pmatrix}
\begin{pmatrix}
x &0 \\
0& x^{-1}
\end{pmatrix}
=
\begin{pmatrix}
x &0 \\
0 & x^{-1}
\end{pmatrix}
\begin{pmatrix}
1 &0 \\
xz & 1
\end{pmatrix} \ \text{with} \ x \in \mathbb{R}^*, z \in \mathbb{R}.$$

We will study these factorizations for the random walk \eqref{model0}. The $N$-part in the $AN$ factorization of $(b_n ^{(\delta)})_{n \in \mathbb{N}}$ is a Markov chain. Following the approach of Babillot \cite{Babillot} we obtain the almost sure convergence of the $N$-part in $NA$ factorization towards a random variable distributed according to the invariant probability measure of the $N$-part in $AN$ factorization.
We will compute this invariant probability measure and use it to obtain a discrete version of the Dufresne identity. \\

The Pitman transform appearing in Pitman's theorem, can be inverted using some additional information (Proposition 2.2 (iv) in \cite{BBO}). From this one can infer reconstruction theorems, in the sense that, it is possible to recover an unconditioned Brownian motion from a conditioned one by applying a sequence of inverse Pitman's transform. We establish in our discrete-time context, a reconstruction theorem. \\

Using Lindeberg's theorem, we prove that our discrete-time process \eqref{model} converges in law towards \eqref{SolutionEDSREDA} when $\delta$ goes to $0$ choosing correctly the increments laws. Thus, we recover Matsumoto--Yor theorem and Chhaibi theorem (th. 3.1 in \cite{Reda2}) in the $SL_2$ case. For the proof of the convergence, we compute some asymptotic formulas for the moments of $\log(GIG)$ laws.

\subsection{Organization of the paper}
In the next section we give some formulas which will be used several times in the paper. In section 3, we prove that the process $(Z_n)_{n \in \mathbb{N}^*}$ is a Markov chain when $\gamma_i$ is distributed according to a GIG law. We give an explicit formula for its Markov kernel. We establish an intertwining relation between this kernel and the transition kernel of $(X_n)_{n \in \mathbb{N}^*}$. In section 4, we prove a new characterization of GIG laws which gives a necessary condition to obtain an intertwining relation between $(Z_n)_{n \in \mathbb{N}^*}$ and $(X_n)_{n \in \mathbb{N}^*}$.
In section 5, we describe the invariant probability measure for the $N$-part in $AN$ factorization of the random walk from which we deduce a discrete-time Dufresne identity. Section 6 is devoted to the reconstruction theorem. In section 7, we establish that our random walk converges towards the continuous process described in the introduction. Finally, in section 8, we use group theory to describe our random walk and prove a convergence theorem for the $N$-part in $NA$ factorization of the random walk.

\bigskip
    {\it Acknowledgments:} I would like to thank Philippe Biane and Manon Defosseux for very helpful discussions and useful comments.

\section{Preliminaries}\label{preliminaires}

\subsection{Explicit expression of the random walk}
Let $(\gamma_n)_{n \in \mathbb{N}}$ be a family of identically distributed independent random variables with $\mathcal{C}^1$ density function and $(X_n)_{n \in \mathbb{N}^*}$, $(Z_n)_{n \in \mathbb{N}^*}$ be two discrete-time processes defined by the random walk $(b_{n} ^{(\delta)})_{n \in \mathbb{N}}$, with $\delta \in \mathbb{R}^*$ a deterministic parameter:
\begin{equation} \label{model}
b_{n+1} ^{(\delta)} := b_{n} ^{(\delta)} g_n ^{(\delta)}  \ \text{with} \ b_{0} ^{(\delta)} = I_2, \ \text{hence} \ b_{n} ^{(\delta)} = g_{n-1} ^{(\delta)} \cdots  g_{0} ^{(\delta)}
\end{equation}
where,
$$b_{n} ^{(\delta)} := \begin{pmatrix}
X_{n} &0 \\
Z_{n} & X_{n} ^{-1}
\end{pmatrix} \ \text{and} \ g_n ^{(\delta)} := \begin{pmatrix}
\gamma_{n} & 0 \\
\delta & \gamma_{n}  ^{-1}
\end{pmatrix}.$$
By iteration of the relation \eqref{model}, we have the following formula for $b_n ^{(\delta)}$, with $n \in \mathbb{N}$,

\begin{equation} \label{ExpressionRW}
b_n ^{(\delta)} 
=
\begin{pmatrix}
\prod_{i=0}^{n-1} \gamma_i &0 \\
\delta \sum_{k=0}^{n-1} \prod_{i=0}^{k-1} \gamma_i ^{-1} \prod_{j=k+1}^{n-1} \gamma_j & \prod_{i=0}^{n-1}  \gamma_i ^{-1}
\end{pmatrix}.
\end{equation}

\noindent
Let us remark that $b_n ^{(\delta)}$ is conjugate to $b_n ^{(1)}$, indeed, for $\delta >0$: 
$$P b_n ^{(\delta)} P^{-1} = b_n ^{(1)} \ \text{with} \ P := \begin{pmatrix}
\delta ^{\frac{1}{2}} &0 \\
0 & \delta ^{- \frac{1}{2}}
\end{pmatrix}.$$ 
In the following, apart from section 6, we set $\delta = 1$ to simplify some expressions and we will denote $b_n$ instead of $b_n ^{(1)}$ and $g_n$ instead of $g_n ^{(1)}$. This slightly changes the expressions of density functions that we will give but it does not change the Markov property and the intertwining relation that we prove in next sections.  

From the expression \eqref{ExpressionRW}, note that we have for $n \in \mathbb{N}^*$: 
\begin{equation}\label{ExpressiondeXetZ}
X_n = \prod_{i=0}^{n-1} \gamma_i \ \text{and} \ Z_n = \sum_{k=0}^{n-1} \left( \prod_{i=0}^{k-1}  \gamma_i ^{-1} \right) \left( \prod_{j=k+1}^{n-1} \gamma_j \right),  \ \text{with} \ X_1 = \gamma_0 \ \text{and} \ Z_1= 1.
\end{equation}

From \eqref{model}, the processes $(Z_n)_{n \in \mathbb{N}^*}$ and $(X_n)_{n \in \mathbb{N}^*}$ satisfy for $k \geq 2$:
\begin{equation} \label{RecurrenceFondamentale}
 Z_k = \frac{X_k Z_{k-1} + 1}{X_{k-1}} \ \text{and} \ X_k = \gamma_{k-1} X_{k-1}.
\end{equation}

\subsection{A change of variable}
We consider the transformation, for $n \geq 2$: 
\begin{equation}\label{Changementdevar}
\begin{array}{rcl}
\Phi_n : \left( \mathbb{R}_+ ^* \right)^n \ \  &\to& \ \ \left( \mathbb{R}_+ ^* \right)^n \\
(y_0, \dots, y_{n-1}) &\mapsto & (z_2, \dots, z_n, x_n)
\end{array}
\end{equation}
where $x_n$ and $z_k$ are defined by: 
\begin{equation}\label{ExpressiondesXetZnonAleatoires}
x_n = \prod_{i=0}^{n-1} y_i \ \text{and} \ z_k = \sum_{l=0}^{k-1} \left( \prod_{i=0}^{l-1}  y_i ^{-1} \right) \left( \prod_{j=l+1}^{k-1} y_j \right).
\end{equation}
They satisfy the following recurrence formula, for $k \geq 2$:
\begin{equation}\label{RecurrenceFondamentale2}
z_k = \frac{x_k z_{k-1} + 1}{x_{k-1}} \ \text{and} \ x_k = y_{k-1} x_{k-1}.
\end{equation}

\begin{lemma}
The transformation $\Phi_n:  \left( \mathbb{R}_+ ^* \right)^n    \to   \left( \mathbb{R}_+ ^* \right)^n $ is a $C^1$-diffeomorphism and the Jacobian, in terms of $z$ variables, associated to this transformation is given by
\begin{equation} \label{Jacobien}
\det ( \mathcal{J}_n ) =  (-1)^{n-1} z_2 \cdots z_n.
\end{equation}
\end{lemma}

\begin{proof}
To prove that $\Phi_n$ is indeed a $C^1$-diffeomorphism and compute the Jacobian, it is more convenient to consider intermediate changes of variables.  
First of all, we consider the change of variables with the expression \eqref{ExpressiondesXetZnonAleatoires}:
$$(y_0, \dots, y_{n-1}) \overset{\widetilde{\Phi}_0}{\mapsto} (x_1, \dots, x_n)$$
which is clearly a $C^1$-diffeormorphism because $x_i$ are products of $y_i$ and the Jacobian of this change of variables is equal to $\prod_{i=1}^{n-1} x_i$. Then successively, according to the expressions \eqref{ExpressiondesXetZnonAleatoires}, we proceed to the following changes of variables 
$$(x_1, \dots, x_n) \overset{\widetilde{\Phi}_1}{\mapsto} (z_2, x_2, \dots, x_n) \overset{\widetilde{\Phi}_2}{\mapsto} (z_2, z_3, x_3, \dots, x_n) \overset{\widetilde{\Phi}_3}{\mapsto} \cdots \overset{\widetilde{\Phi}_{n-1}}{\mapsto} (z_2, \dots, z_n, x_n).$$
Thanks to formula \eqref{RecurrenceFondamentale2}, each change of variable $\widetilde{\Phi}_k$ is a $C^1$-diffeomorphism with Jacobian equal to, for all $k \in \llbracket 1, n-1 \rrbracket$: 
$$\frac{\partial z_{k+1}}{\partial x_{k}} = \frac{-(x_{k+1} z_{k} + 1)}{x_{k} ^2}.$$
The resulting change of variable is just $\Phi_n = \widetilde{\Phi}_{n-1} \circ \cdots \circ \widetilde{\Phi}_1 \circ \widetilde{\Phi}_0$, so the corresponding Jacobian is given by
\begin{equation*} 
\det ( \mathcal{J}_n ) = \prod_{i=1}^{n-1}  x_k  \prod_{k=1}^{n-1} \frac{-( x_{k+1} z_k +1) }{x_k ^2} = (-1)^{n-1} \prod_{k=1}^{n-1} \left( \frac{x_{k+1} z_k +1}{x_k} \right) =  (-1)^{n-1} z_2 \cdots z_n. 
\end{equation*}
Where the last equality is obtained again from \eqref{RecurrenceFondamentale2}. 
\end{proof}

Let us mention that the bijection $\Phi_n$ is related to a geometric version of the Robinson--Schensted--Knuth (RSK) correspondence. We refer to \cite{Corwin} and \cite{Oconnell} for more details about this correspondence.

\section{Markov property of the discrete-time Matsumoto--Yor process $(Z_n)_{n \in \mathbb{N}^*}$ and intertwining relation }

In this section, we consider Generalized Inverse Gaussian distributions $GIG(\lambda,a, b)$. \\
These distributions were introduced by Halphen in \cite{Halphen} who used these laws in hydrological problems. They have also been used by Good in his study of population frequencies \cite{Good}.\\

Let us recall that the probability density function for a law $GIG(\lambda, a, b)$ is given by
$$x \in \mathbb{R}_+ ^* \mapsto \left( \frac{b}{a} \right)^{\lambda} \frac{1}{2 K_{\lambda} (ab)} x^{\lambda -1} e^{- \frac{1}{2} (\frac{a^2}{x} + b^2 x)} $$
where $a,b >0$, $\lambda \in \mathbb{R}$ and $K_{\lambda}$ is a Macdonald function. An integral representation of this function is given for $z >0$ by
$$ K_{\lambda} (z) = \frac{1}{2} \int_{0}^{+ \infty}  x^{\lambda -1} e^{- \frac{z}{2} (x+ \frac{1}{x})} dx.$$
Let us recall the scaling property between GIG distributions: 
$$\text{If} \ c >0 \ \text{and} \ X \sim  GIG(\lambda, a, b), \ \text{then} \ cX \sim GIG \left( \lambda, a \sqrt{c}, \frac{b}{\sqrt{c}} \right).$$ 
In this paper, we focus on the $GIG( \lambda, a ,a)$ distribution with the probability density function
$$x \in \mathbb{R}_+ ^* \mapsto \frac{1}{2 K_{\lambda} (a^2)} x^{\lambda -1} e^{- \frac{a^2}{2} (x+ \frac{1}{x})}.$$

In the following, we prove that the process $(Z_n)_{n \in \mathbb{N}^*}$ is a Markov chain when $(\gamma_n)_{n \in \mathbb{N}}$ is distributed according to the law $GIG( \lambda, a, a)$ with $a>0$ and $\lambda \in \mathbb{R}$.
We compute the transition kernel of $(Z_n)_{n \in \mathbb{N}^*}$ and as a by-product, we establish an intertwining relation between this kernel and the transition kernel of $(X_n)_{n \in \mathbb{N}^*}$. \\

The following expression is crucial for the proof of Markov property, and it is obtained by induction on $n \geq 2$.

\begin{lemma}\label{GammaenfonctiondesX}
We have the following equality 
$$\forall n \geq 2, \ \  \sum_{k=0}^{n-1} \left( \gamma_k + \frac{1}{\gamma_k} \right)= \frac{1}{Z_n} \left( X_n + \frac{1}{X_n} \right) + F_n ( Z_2, \dots, Z_n) $$
where $F_n ( Z_2, \dots, Z_n) = \sum_{k=1}^{n-1} \frac{Z_{k+1} ^2 + Z_k ^2 +1}{Z_{k+1} Z_k}$.
\end{lemma}

From the expression above, we deduce the following joint density functions.

\begin{proposition}\label{Loiduvecteur}
The joint density function of $(Z_2, \dots, Z_n, X_n)$ is, for $z_2, \dots, z_n, x_n >0$:

\begin{equation} \label{distribZetX}
f_{(Z_2, \dots, Z_n, X_n)} (z_2, \dots, z_n, x_n)= \left( \frac{1}{ 2 K_{\lambda} (a^2)} \right)^n \frac{ x_n ^{\lambda - 1}}{z_2 \cdots z_n} e^{- \frac{a^2}{2 z_n} (x_n+ \frac{1}{x_n}) - \frac{a^2}{2} F_n ( z_2, \dots, z_n)}.
\end{equation}

The joint density function of $(Z_2, \dots, Z_n)$ is, for $z_2, \dots, z_n >0$:
\begin{equation} \label{distribZ}
f_{(Z_2, \dots, Z_n)} (z_2, \dots, z_n)= \left( \frac{1}{ 2 K_{\lambda} (a^2)} \right)^n \frac{2 K_{\lambda} \left( \frac{a^2}{z_n} \right)}{z_2 \cdots z_n} e^{ - \frac{a^2}{2} F_n ( z_2, \dots, z_n)}.
\end{equation} 
The conditional density function of $X_n$ given $\left\{ Z_n = z_n, \dots, Z_2 = z_2 \right\}$ is, for $z_2, \dots, z_n >0$:
\begin{equation} \label{distribXsachantZ}
f_{X_n |  Z_n = z_n, \dots, Z_2 = z_2} (x_n) = f_{X_n |  Z_n = z_n} (x_n) = \frac{1}{2 K_{\lambda} \left( \frac{a^2}{z_n} \right)} x_n ^{\lambda -1} e^{- \frac{a^2}{2z_n} \left( x_n + \frac{1}{x_n} \right)}.
\end{equation}
\end{proposition}
\begin{proof}
Let $h: \mathbb{R}^n \to \mathbb{R}$ be a bounded continuous function and $f$ the common density function of $\gamma_i$'s. Thanks to the independence of $\gamma_i$'s and the change of variables $\Phi_n$ defined by \eqref{Changementdevar} we get:
\begin{align*}
 \mathbb{E} \left( h(Z_2, \dots, Z_n, X_n ) \right) &= \int_{\left( \mathbb{R}_+ ^* \right)^n } h( \Phi_n( y_0, \dots, y_{n-1})) f(y_0) \cdots f(y_{n-1}) dy_0 \cdots d y_{n-1} \\
&= \int_{\left( \mathbb{R}_+ ^* \right)^n } h( \Phi_n( y_0, \dots, y_{n-1}))  \frac{\left( \prod_{i=0}^{n-1} y_i \right)^{\lambda -1}}{ \left( 2 K_{\lambda} (a^2) \right)^n}  e^{ - \frac{a^2}{2} \sum_{k=0}^{n-1} \left( y_k + \frac{1}{y_k} \right)} dy_0 \cdots d y_{n-1}.  
\end{align*}
From \eqref{Jacobien} and Lemma \ref{GammaenfonctiondesX} we deduce the joint density of $(Z_2, \dots, Z_n, X_n)$. \\
Integrating with respect to the variable $x_n$ in the formula \eqref{distribZetX} we obtain the density \eqref{distribZ} using integral representation of the Macdonald function. \\
The equality, 
$$f_{X_n |  Z_n = z_n, \dots, Z_2 = z_2} (x_n) = \frac{1}{2 K_{\lambda} \left( \frac{a^2}{z_n} \right)} x_n ^{\lambda -1} e^{- \frac{a^2}{2z_n} \left( x_n + \frac{1}{x_n} \right)}$$ 
is obtained from the quotient of \eqref{distribZetX} and \eqref{distribZ}. This last conditional density only depends on $x_n$ and $z_n$, so we obtain
$$f_{X_n |  Z_n = z_n, \dots, Z_2 = z_2} (x_n) = f_{X_n |  Z_n = z_n} (x_n).$$
\end{proof}

\begin{proposition} \label{ZKernel}
The process $(Z_n)_{n \in \mathbb{N}^*}$ is a homogeneous Markov chain starting from $Z_1 = 1$ with transition kernel given, for $x >0$, by
$$ Q(x,dy) = \left( \frac{1}{2 K_{\lambda} (a^2)} \right) \frac{K_{\lambda} \left(\frac{a^2}{y} \right)}{K_{\lambda} \left(\frac{a^2}{x} \right)} \frac{1}{y} e^{- \frac{a^2 (x^2 + y^2 + 1)}{2xy}} \mathds{1}_{\mathbb{R}_+ ^*} (y) dy. $$
\end{proposition}
\begin{proof}
Using the density function \eqref{distribZ} at rank $n$ and $n-1$, we obtain the conditional density function for $z_2, \dots, z_n >0$:
\begin{equation} \label{conditionnelledesZ}
f_{Z_n | (Z_{n-1}, \dots, Z_2 )= ( z_{n-1}, \dots, z_2 )  } (z_n) = \left( \frac{1}{2 K_{\lambda} (a^2)} \right) \frac{K_{\lambda} \left(\frac{a^2}{z_n} \right)}{K_{\lambda} \left(\frac{a^2}{z_{n-1}} \right)} \frac{1}{z_{n}} e^{- \frac{a^2 (z_n ^2 + z_{n-1} ^2 + 1)}{2z_n z_{n-1}}}.
\end{equation}
This last expression only depends on $z_n$ and $z_{n-1}$ this implies the Markov property for the process $(Z_n)_{n \in \mathbb{N}^*}$ and the expression for the transition kernel.
\end{proof}

One can also give the transition kernel for the Markov chain $(X_n)_{n \in \mathbb{N}^*}$ starting from $X_1 = \gamma_0$, for $x>0$:
$$P(x,dy) = \left( \frac{1}{2 K_{\lambda} (a^2)} \right) \frac{y^{\lambda - 1}}{x^\lambda} e^{- \frac{a^2}{2} \left( \frac{y}{x} + \frac{x}{y} \right)}  \mathds{1}_{ \mathbb{R}_+ ^*} (y) dy.$$

We will establish an intertwining relation between the transition kernels of $(X_n)_{n \in \mathbb{N}^*}$ and $(Z_n)_{n \in \mathbb{N}^*}$. Let us recall the definition: 

\begin{definition}
Let $P$ and $Q$ be two Markov transition kernels on the measurable spaces $(E, \mathcal{E})$ and $(F, \mathcal{F})$ respectively. A Markov kernel $\Lambda$ from $F$ to $E$ is a map 
$$\Lambda : (u, A) \mapsto \Lambda(u,A) \ \text{with} \ u \in F \ \text{and} \ A \in \mathcal{E} $$
such that for each $u \in F$, $\Lambda(u, \cdot)$ is a probability on $E$, and for each $A \in \mathcal{E}$, $\Lambda( \cdot, A)$ belongs to the space of bounded measurable functions on $F$. \\
Then, the Markov kernel $\Lambda$ intertwines $P$ and $Q$ if one has the relation: 
$$\Lambda P = Q \Lambda $$
where the composition of kernels is defined by $\Lambda P (u, dv) := \int_{E} P(y,dv) \Lambda (u,dy)$.
\end{definition}

We refer to \cite{BianeIntertwining} and \cite{CPY} for some examples of intertwining relations. \\
Let us recall that Rogers--Pitman criterion \cite{RogersPitman} gives conditions for a process that is a function of a Markov process to be Markov itself. One can use this criterion here to deduce the Markov property of the process $(Z_n)_{n \in \mathbb{N}^*}$ by giving the intertwining kernel $\Lambda$. This intertwining kernel is given, for all $n \in \mathbb{N}^*$, by:
$$ \mathbb{P} \left( X_n \in dx | Z_n , \cdots, Z_1 \right) = \Lambda(Z_n, dx) \ \text{a.s.}$$
Here, from Proposition \ref{Loiduvecteur}, the conditional law $\mathcal{L}(X_n | Z_n , \cdots, Z_1 )$ is in fact $\mathcal{L}(X_n | Z_n )$. We recover the density function of a GIG law in \eqref{distribXsachantZ} and we obtain:

\begin{proposition}\label{entrelacement}
The Markov transition kernels of the processes $(X_n)_{n \in \mathbb{N}^*}$ and $(Z_n)_{n \in \mathbb{N}^*}$ are intertwined by the Markov kernel $\Lambda$ given by the law of $X_n$ given $Z_n$ which does not depend on $n$, more precisely, we have the following intertwining relation:
$$\Lambda P = Q \Lambda $$
where, for $z >0$: 
$$\Lambda (z, dx) =  \frac{1}{2 K_{\lambda} \left( \frac{a^2}{z} \right)} x^{\lambda -1} e^{- \frac{a^2}{2z} \left( x + \frac{1}{x} \right)} \mathds{1}_{\mathbb{R}_+ ^* } (x) dx.$$
 \end{proposition}
\begin{proof}
We deduce the formula of $\Lambda$ from \eqref{distribXsachantZ}. 
Then, one has for $z>0$:
$$\Lambda P (z, dx) = \int_{\mathbb{R}_+ ^*} \frac{1}{4 K_{\lambda} \left( a^2 \right) K_{\lambda} \left( \frac{a^2}{x} \right)} \frac{x^{\lambda -1 }}{y} e^{- \frac{a^2}{2} \left( \frac{x}{y} + \frac{y}{x} + \frac{y}{z} + \frac{1}{zy} \right) } \mathds{1}_{ \mathbb{R}_+ ^*}  (x) dy dx.$$
Moreover, we have for $z>0$:
\begin{align*}
Q \Lambda (z,dx) &= \int_{\mathbb{R}_+ ^*} \Lambda(y,dx) Q (z,dy) \\
&= \int_{\mathbb{R}_+ ^*} \frac{1}{4 K_{\lambda} \left( a^2 \right) K_{\lambda} \left( \frac{a^2}{x} \right)} \frac{x^{\lambda -1 }}{y} e^{- \frac{a^2}{2} \left( \frac{x}{y} + \frac{1}{yx} + \frac{z}{y} + \frac{y}{z} + \frac{1}{zy} \right) } \mathds{1}_{ \mathbb{R}_+ ^*} (x) dy dx \\
&= \Lambda P (z, dx)
\end{align*}
where the last equality is obtained by the change of variable $y := \frac{1}{u} ( 1 + zx)$. 
\end{proof}

Let us mention that the intertwining relation can also be deduced from the work of Chhaibi \cite{Reda} (th. 5.6.8).

\section{A characterization of Generalized Inverse Gaussian distributions (GIG)}

There exist several characterizations of GIG distributions, see for example the work of Koudou, Ley, Letac, Seshadri, Matsumoto, Vallois and Yor in \cite{Koudou} \cite{Letac} \cite{Letac2}  \cite{MY2} \cite{Vallois}. A survey can be found in the paper \cite{KoudouLey}. 
We give in the following proposition a new characterization for these laws involving $\gamma_0$, $\gamma_1$ and $\gamma_2$. Let us recall that from \eqref{ExpressiondeXetZ} we get:
$$X_2 = \gamma_0 \gamma_1, \ X_3 = \gamma_0 \gamma_1 \gamma_2, \ Z_2 = \gamma_0 ^{-1} + \gamma_1, \ Z_3 = \gamma_0 ^{-1} \gamma_1 ^{-1} + \gamma_0 ^{-1} \gamma_2  +  \gamma_1 \gamma_2.$$
\begin{proposition}\label{TheoremeNecessaire}
Let $\gamma_0, \gamma_1, \gamma_2$ be three i.i.d. random variables with $\mathcal{C}^1$ density function supported on $\mathbb{R}_{+}$. If we have the equality of conditional laws, for $z,u >0$: 
$$\mathcal{L} \left( X_3  |  Z_3=z, Z_2 = u \right) = \mathcal{L} \left( X_2 | Z_2=z \right),$$
then, $\gamma_0$ is distributed according to the law $GIG( \lambda, a, a)$ with some $a>0$ and $\lambda \in \mathbb{R}$.
\end{proposition}

\begin{proof}
Assume that  $\gamma_0, \gamma_1, \gamma_2$ are i.i.d. random variables with a $\mathcal{C}^1$ density supported on $\mathbb{R}_{+}$ denoted by $f$. If the equality of conditional laws holds, then conditional density functions satisfy:
$$f_{X_2 | Z_2 = z} = f_{X_3 | (Z_3, Z_2 )=(z,u)}.$$
We have the following expressions by inverting $\Phi_2$ from \eqref{Changementdevar}: 
$$\gamma_0 = \frac{X_2 +1}{Z_2} \ \text{and} \ \gamma_1 = \frac{X_2 Z_2}{X_2 +1}.$$ 
Moreover we know from \eqref{Jacobien} that $ | \det(\mathcal{J}_2) | = \frac{1}{z}$ for the change of variable $\Phi_2$. We get the joint density, 
$$f_{(Z_2, X_2)} (z,x) = \frac{1}{z} f \left( \frac{x+1}{z} \right) f \left( \frac{xz}{x+1} \right)$$ 
and hence we deduce the conditional density
$$f_{X_2 | Z_2 = z} ( x ) = \frac{ f \left( \frac{x+1}{z} \right) f \left( \frac{xz}{x+1} \right)}{I(z)}$$
where 
$$I(z):= \int_{\mathbb{R}}  f \left( \frac{t+1}{z} \right) f \left( \frac{tz}{t+1} \right) dt$$ 
depends only on the variable $z$. \\
In the same way, we know by inverting $\Phi_3$ from \eqref{Changementdevar} that
$$\gamma_0 = \frac{Z_2 X_3 + Z_3 + 1}{Z_2 Z_3}, \ \gamma_1 = \frac{Z_2 ^2 X_3 + Z_2}{Z_2 X_3 +Z_3 +1}, \ \gamma_2 = \frac{X_3 Z_3}{Z_2 X_3 +1}$$ 
and from \eqref{Jacobien} we obtain $ | \det(\mathcal{J}_3) | = \frac{1}{zu}$ for the change of variable $\Phi_3$, we deduce 
$$f_{X_3 | (Z_3, Z_2 )=(z,u)} (x)= \frac{f\left( \frac{ux+z+1}{uz} \right) f \left( \frac{u^2 x +u}{ux+z+1}\right) f \left( \frac{xz}{ux+1} \right)}{ J(z,u)}$$
where
$$J(z,u) := \int_{\mathbb{R}} f\left( \frac{ut+z+1}{uz} \right) f \left( \frac{u^2 t +u}{ut+z+1}\right) f \left( \frac{tz}{ut+1} \right) dt $$ 
depends only on the variables $z$ and $u$. \\
Thus, the equation $f_{X_2 | Z_2 = z} (x) = f_{X_3 | (Z_3, Z_2 )=(z,u)} (x)$ implies that 
$$\frac{J(z,u)}{I(z)} =  \frac{f\left( \frac{ux+z+1}{uz} \right) f \left( \frac{u^2 x +u}{ux+z+1}\right) f \left( \frac{xz}{ux+1} \right)}{f \left( \frac{x+1}{z} \right) f \left( \frac{xz}{x+1} \right)} $$ 
does not depends on the variable $x$. The logarithmic derivative in $x$ of this last quotient is therefore equal to $0$. Let us denote by $g := \frac{f'}{f}$ the logarithmic derivative of $f$, we obtain the following functional equation:
\begin{equation} \label{EF}
\begin{split}
\frac{1}{z} g \left( \frac{ux+z+1}{uz} \right) + \frac{u^2 z}{(ux+z+1)^2} g \left( \frac{u^2 x + u}{ux+z+1} \right) &+ \frac{z}{(ux+1)^2} g \left( \frac{xz}{ux+1} \right) \\
 &- \frac{1}{z} g \left( \frac{x+1}{z} \right) - \frac{z}{(x+1)^2} g \left( \frac{xz}{x+1} \right) = 0.
\end{split}
\end{equation}
Let $z := \frac{x+1}{x}$ and $u = \frac{1}{x^2}$, the above equation becomes:
\begin{equation} \label{EF4}
x^2 g(2x^2) + \frac{1}{4 x^2} g \left( \frac{1}{2 x^2} \right) = g(1).
\end{equation}
Setting $s := 2 x^2$, we get: 
\begin{equation} \label{EF5}
s g(s) + \frac{1}{s} g \left( \frac{1}{s} \right) = 2 g(1).
\end{equation}
By the change of function $G(s) := s g(s)$, we obtain the following equation: 
\begin{equation} \label{EF6}
G(s) + G \left( \frac{1}{s} \right) = 2 G(1).
\end{equation}
The solutions of \eqref{EF6} are given by $G(x)= G(1) + \left( x - \frac{1}{x} \right) \varphi(x)$ where $\varphi$ is some continuous function satisfying $\varphi(x)= \varphi \left( \frac{1}{x} \right)$. Thus, the solutions of \eqref{EF} are of the form: 
\begin{equation} \label{EF7}
g(x) = \frac{g(1)}{x} + \left(1 - \frac{1}{x^2} \right) \varphi(x) \ \ \text{with} \ \varphi(x)= \varphi \left( \frac{1}{x} \right).
\end{equation}
We will prove that the only solutions of \eqref{EF} correspond to the case $\varphi$ constant.
Indeed, if $g$ is solution of \eqref{EF}, then $h(x) := g \left( \frac{1}{x} \right)$ is solution of: 

\begin{equation} \label{EF8b}
\begin{split}
\frac{1}{z} h \left( \frac{uz}{ux+z+1} \right) + \frac{u^2 z}{(ux+z+1)^2} h \left( \frac{ux+z+1}{u^2 x + u} \right) &+ \frac{z}{(ux+1)^2} h \left( \frac{ux+1}{xz} \right) \\
 &- \frac{1}{z} h \left( \frac{z}{x+1} \right) - \frac{z}{(x+1)^2} h \left( \frac{x+1}{xz}\right) = 0.
\end{split}
\end{equation}
First, let us set $x = \frac{1}{2y}$, $u= 2 y^2 - y$ and $z= y + \frac{1}{2}$. We get:

\begin{equation} \label{EF9}
\frac{1}{y^2} h \left( \frac{2y^2 - y}{2} \right) + \left( \frac{2y-1}{2} \right)^2 h \left( \frac{2}{2y^2 - y} \right) + \frac{1}{y^2} h(2y) - \frac{1}{y^2} h(y)= h(2).
\end{equation}

Since it has been proved that the solutions of \eqref{EF} are of the form \eqref{EF7}, the function $h(x)= x h(1) + \left( 1 - x^2 \right) \varphi(x)$ where $\varphi(x)= \varphi \left( \frac{1}{x} \right)$ must satisfy \eqref{EF9}. Substituting this expression into the equation \eqref{EF9} where we replaced the letter $y$ by the letter $x$, leads to the functional equation in $\varphi$: 

\begin{equation} \label{EF10}
( 4 x^2 - 1) \varphi(2x) + (1-x^2) \varphi(x)= 3x^2 \varphi(2) \ \  \text{and} \ \ \varphi(x)= \varphi \left( \frac{1}{x} \right).
\end{equation}

We will prove in Lemma \ref{EqFon} that the only continuous solutions of \eqref{EF10} are given by $\varphi$ constant.
To conclude, according to \eqref{EF7} this proves that the only continuous solutions of \eqref{EF} are given by
$$g(x)= \frac{C_1}{x} + C_2 \left( 1 - \frac{1}{x^2} \right)$$ where $C_1, C_2 \in \mathbb{R}$.
Let us recall that $g := \frac{f'}{f}$. The solutions of the ordinary differential equation
$$f'(x)= \left( \frac{C_1}{x} + C_2 \left( 1 - \frac{1}{x^2} \right) \right) f(x)$$
are given by $f(x)=K x^{C_1} e^{C_2 \left( x + \frac{1}{x} \right)}$ where $K \in \mathbb{R}$, so we recover the density function of GIG law.
\end{proof}

\begin{theorem}\label{Caractérisation}
Let $(\gamma_n)_{n \in \mathbb{N}}$ be a sequence of i.i.d. random variables with $\mathcal{C}^1$ density function supported on $\mathbb{R}_{+}$. Then, there exists an intertwining kernel $\Lambda$  such that, for all $n \in \mathbb{N}^*$:
$$ \mathbb{P} \left( X_n \in dx | Z_n , \cdots, Z_1 \right) = \Lambda(Z_n, dx) \ \text{a.s.}$$
if and only if the family $(\gamma_n)_{n \in \mathbb{N}}$ is distributed according to the law $GIG( \lambda, a, a)$ with some $a>0$ and $\lambda \in \mathbb{R}$. 
\end{theorem}
\begin{proof}
Proposition \ref{Caractérisation} gives the necessary condition for the existence of an intertwining kernel. The converse was obtained in Proposition \ref{entrelacement} where we considered the law $GIG( \lambda, a ,a)$ for the increments.
\end{proof}

\section{Invariant probability measure for the $N$-part and Dufresne identity} \label{soussection2}

We consider the $N$-part of the random walk in both $NA$ and $AN$ factorizations. In \ref{soussection1} we will explain the link between the two. Then, in \ref{soussection2} we will obtain the invariant probability measure of the $N$-part in the $AN$ factorization. Thanks to this, we will obtain a Dufresne identity involving the $N$-part of the $NA$ factorization. In this section, several results obtained remain true more generally for any random walk on group with $NA$ factorization. For the convenience of the reader, we prove the results for $\delta = 1$.

\subsection{Relations between both factorizations}
\label{soussection1}

Here, we only assume that $\gamma := (\gamma_i)_{i \geq 0}$ is a sequence of i.i.d. random variables with finite first $\log$-moment. Let $\mathcal{N}_n ( \gamma )$ be the $N$-part of the random walk in the $NA$ factorization and $\widetilde{\mathcal{N}}_n ( \gamma)$ be the $N$-part of the random walk in the $AN$ factorization.
That is to say, with the notation $b_n (\gamma)$ instead of $b_n$ defined in section \ref{preliminaires}, we have
$$b_n  (\gamma)=  \mathcal{N}_n ( \gamma ) A_n ( \gamma ) = A_n ( \gamma ) \widetilde{\mathcal{N}}_n ( \gamma) \ \text{where} \ A_n ( \gamma ) :=  \begin{pmatrix}
X_n & 0 \\
0 & X_n ^{-1}
\end{pmatrix}.$$
Then, we denote by $N_n (\gamma)$ and $\widetilde{N}_n (\gamma) $ the matrix coefficients in the matrices:
$$\mathcal{N}_n ( \gamma )  = \begin{pmatrix}
1 & 0 \\
N_n (\gamma) & 1
\end{pmatrix} \ \ \text{and} \ \ \widetilde{\mathcal{N}}_n (\gamma) = \begin{pmatrix}
1 & 0\\
\widetilde{N}_n (\gamma) & 1
\end{pmatrix} .$$
We obtain the following expressions for $n \in \mathbb{N}^*$:
$$N_n (\gamma) = X_n ^{-1} Z_n \ \text{and} \ \widetilde{N}_n (\gamma) = X_n Z_n$$

We will see that the $N$-part in $AN$ factorization $( \widetilde{\mathcal{N}}_n (\gamma) )_{n \in \mathbb{N}^*}$ is a Markov chain. In general, $\left( N_n  (\gamma) \right)_{n \in \mathbb{N}^*}$ the $N$-part in $NA$ factorization is not a Markov chain, but we have the following relation which is a classic tool to prove Dufresne-type identities, see \cite{AristaOconnell} and \cite{ChamayouLetac} for instance. 
We prove this relation in our context.

\begin{lemma} We have the equality in law: 
$$\forall n \in \mathbb{N}, \   N_n (\gamma) \overset{\text{law}}{=} \widetilde{N}_n \left( \gamma ^{-1} \right).$$
\end{lemma}
\begin{proof} 
The increments of $(b_n (\gamma))_{n \in \mathbb{N}}$ are denoted by $(g_n ( \gamma))_{n \in \mathbb{N}}$ and we recall that for $k \in \mathbb{N}$:
$$g_k  ( \gamma) := \begin{pmatrix}
\gamma_k  & 0 \\
1 & \gamma_k ^{-1}
\end{pmatrix}.$$
Let $w_0$ be the matrix
$$w_0 := \begin{pmatrix}
0 & 1 \\
1 & 0
\end{pmatrix}.$$
Since $w_0 A_n (\gamma) w_0 ^{-1} = A_n ( \gamma ^{-1} )$ and $ w_0 \left[ \mathcal{N}_n ( \gamma ) \right]^t w_0^{-1} = \mathcal{N}_n ( \gamma )$, using $NA$ factorization we get
\begin{equation}\label{Equation1}
w_0 \left[ b_n ( \gamma ) \right]^t w_0 ^{-1} =  A_n ( \gamma ^{-1} ) \mathcal{N}_n ( \gamma ).
\end{equation}
From $AN$ factorization, we have
\begin{equation}\label{Equation2}
b_n ( \gamma ^{-1} ) =  A_n ( \gamma^{-1} ) \widetilde{\mathcal{N}}_n ( \gamma ^{-1}).
\end{equation}
Moreover, for all $k \in \mathbb{N}$,
$$w_0 [ g_k ( \gamma) ]^t w_0^{-1} = \begin{pmatrix}
\gamma_k ^{-1} & 0 \\
1 & \gamma_k
\end{pmatrix} = g_k ( \gamma ^{-1}).$$
From
$$w_0 \left[ b_n  ( \gamma ) \right]^t w_0 ^{-1} = w_0 \left[ g_{n-1} ( \gamma )  \right]^t w_0 ^{-1} \cdots w_0 \left[ g_{0} ( \gamma ) \right]^t w_0 ^{-1} = g_{n-1} ( \gamma ^{-1}) \cdots g_0 ( \gamma ^{-1}),$$
using that $\gamma$ is a sequence of i.i.d. random variables,
$$g_{n-1} ( \gamma ^{-1}) \cdots g_0 ( \gamma ^{-1}) \overset{\text{law}}{=} g_{0} ( \gamma ^{-1}) \cdots g_{n-1} ( \gamma ^{-1})$$
we deduce for all $n \in \mathbb{N}$, 
$$w_0 \left[ b_n  ( \gamma ) \right]^t w_0 ^{-1} \overset{\text{law}}{=} b_n  \left( \gamma ^{-1} \right).$$
Hence, using the equality in law between \eqref{Equation1} and \eqref{Equation2}, we obtain the equality $N_n (\gamma) \overset{\text{law}}{=} \widetilde{N}_n \left( \gamma ^{-1} \right)$ from unicity of the $AN$ factorization.
\end{proof}

In the following, we study the convergence of the process $\left( N_n ( \gamma) \right)_{n \in \mathbb{N}^*} $. The next proposition gives the asymptotic distribution for this process and will be proved in section \ref{SectionCVNpart}. 

\begin{proposition}\label{PropBabillot}
Assume that $\mathbb{E} \left( \log \gamma_0 \right) >0$ and that the probability measure of $g_0$ is spread-out\footnote{This condition will be defined in Section \ref{SectionCVNpart}.} on $G$ the subgroup of lower matrices in $SL_2$. The process $\left( N_n (\gamma) \right)_{n \in \mathbb{N}^*}$ converges almost surely, and the law of $N_{\infty} (\gamma) := \lim_{n \to + \infty} N_n (\gamma)$ is equal to the unique invariant probability measure of the process $( \widetilde{N}_n ( \gamma ^{-1} ) )_{n \in \mathbb{N}^*}$. 
\end{proposition}

\begin{rem}\label{RemarqueCVPS}
From the expressions of $X_n$ and $Z_n$ in \eqref{ExpressiondeXetZ} we get for all $n \in \mathbb{N}^*$: 
\begin{equation}\label{ExpressionsdesN}
N_n (\gamma) = \sum_{k=0}^{n-1} \gamma_k ^{-1} \left( \prod_{i=0}^{k-1} \gamma_i ^{-1} \right)^2 \ \text{and} \ \widetilde{N}_n  (\gamma) =  \sum_{k=0}^{n-1} \gamma_k \left( \prod_{j=k+1}^{n-1} \gamma_j \right)^2.
\end{equation}
\noindent
From the expression \eqref{ExpressionsdesN}, we obtain
$$N_n (\gamma) = \sum_{k=0}^{n-1} e^{- \log (\gamma_k) - 2 \sum_{i=0}^{k-1} \log ( \gamma_i ) }.$$
Using the law of large number and an exponential decay argument we deduce the almost sure convergence when $\mathbb{E} \left( \log \gamma_0 \right) >0$.
\end{rem}

\subsection{The case of GIG distributions}
\label{soussection2}

Now, we consider that
$$(\gamma_i)_{i \geq 0} \overset{\text{i.i.d.}}{\sim} GIG( \lambda , a , a ).$$ 
Hence, we have $(\gamma_i ^{-1})_{i \geq 0} \overset{\text{i.i.d.}}{\sim} GIG( - \lambda , a , a )$.
In that case, we will denote $N_n ^{(\lambda)}$ instead of $N_n ( \gamma)$ and $\widetilde{N}_n ^{(-\lambda)}$ instead of $\widetilde{N}_n \left( \gamma ^{-1} \right)$ to simplify the notations. 
As we will see in next lemma, it is important to properly choose the sign of the parameter $\lambda$ to ensure the convergence of the corresponding limiting processes according to Remark \ref{RemarqueCVPS} and Proposition \ref{PropBabillot}.

\begin{lemma}\label{signeduloggamma}
The sign of $\mathbb{E} \left( \log \gamma_0 \right)$ is the same as the sign of the parameter $\lambda$.
\end{lemma}
\begin{proof}
By definition of $GIG( \lambda , a , a )$  density: 
$$\mathbb{E} \left( \log \gamma_0 \right) = \frac{1}{2 K_{\lambda} (a^2)} \int_{0}^{+ \infty} \log (x)  x^{\lambda -1} e^{- \frac{a^2}{2} (x+ \frac{1}{x})} dx.$$
The sign of $\mathbb{E} \left( \log \gamma_0 \right)$ is the same sign as the integral because Macdonald function is positive. By cutting the integral on the two parts $]0, 1]$ and $]1, + \infty[$ and doing the change of variable $x \mapsto \frac{1}{x}$ on the first one, we obtain
$$\mathbb{E} \left( \log \gamma_0 \right) = \frac{1}{2 K_{\lambda} (a^2)} \int_{1}^{+ \infty} \frac{\log (x)}{x}  \left( x^{\lambda} - x^{- \lambda} \right) e^{- \frac{a^2}{2} (x+ \frac{1}{x})} dx.$$
Hence, we deduce the statement of the lemma.
\end{proof}

\begin{rem}\label{LGN1}
When $\lambda = 0$, $\left( X_n \right)_{n \in \mathbb{N}^*}$ is recurrent. Using the law of large numbers, we deduce when $\lambda >0$ (resp. $\lambda <0$) that $X_n \overset{\text{a.s.}}{\to} + \infty \ (\text{resp.} \ 0)$ when $n \to + \infty$. 
\end{rem}

Now, we obtain the distribution of the asymptotic $N$-part of the random walk when $\lambda >0$. As we will see later, $N_{\infty} ^{(\lambda)} := \lim_{n \to + \infty} N_n ^{(\lambda)}$ is an infinite sum of random variables whose law cannot be obtained directly. However, it is sufficient to obtain the invariant probability measure for $(\widetilde{N}_n ^{(-\lambda)} )_{n \in \mathbb{N}^*}$ as we will show. 

\begin{proposition} \label{Mesureinvariante}
The process $( \widetilde{N}_n ^{(-\lambda)} )_{n \in \mathbb{N}^*}$ is a homogeneous Markov chain starting from $\widetilde{N}_1 ^{(-\lambda)} = \gamma_0 ^{-1}$ with transition kernel given for $x>0$ by
$$\widetilde{K}(x,dy) = \frac{1}{2 K_{\lambda} (a^2) \sqrt{1+4xy}} \left( \frac{-1 + \sqrt{1+4xy}}{2x}\right) ^{- \lambda -1} e^{- \frac{a^2}{2} \left( \frac{-1+\sqrt{1+4xy}}{2x} + \frac{2x}{-1+\sqrt{1+4xy}} \right)} \mathds{1}_{ \mathbb{R}_+ ^* } (y)dy .$$
Moreover if $\lambda >0$, the Markov chain $( \widetilde{N}_n ^{(-\lambda)} )_{n \in \mathbb{N}^*}$ is reversible with invariant probability measure
\begin{equation} \label{MesInvariante}
d \pi(x) = \frac{a^{2 \lambda}}{2^{\lambda} \Gamma ( \lambda)} x^{- \lambda - 1} e^{- \frac{a^2}{2x}} \mathds{1}_{\mathbb{R}_{+} ^*} (x)  dx 
\end{equation}
where $\Gamma$ is Euler's gamma function. Thus, $\pi$ is an inverse-gamma distribution.
\end{proposition}
\begin{proof} 
From the recurrence formula \eqref{RecurrenceFondamentale} we obtain for all $n \geq 2$: 
\begin{equation}\label{RecurrencePourNtilde}
 \widetilde{N}_n ^{(-\lambda)} = \gamma_{n-1} ^{-2}  \widetilde{N}_{n-1} ^{(-\lambda)} + \gamma_{n-1} ^{-1}.
\end{equation}
Then, it is clear that $( \widetilde{N}_n ^{(-\lambda)} )_{n \in \mathbb{N}^*}$ is a Markov chain since $\gamma_{n-1} ^{-1}$ is independent of $ \widetilde{N}_{n-1} ^{(-\lambda)}$. \\
The transition kernel $\widetilde{K}$ is obtained from  \eqref{RecurrencePourNtilde} computing the law of $\gamma_0 ^2 x + \gamma_0$ when $\gamma_0$ follows $GIG(-\lambda, a ,a)$ distribution and using the symmetric relation $K_{\lambda} = K_{- \lambda}$. \\
For the reversibility we consider the following calculation, for $x,y >0$:
$$\frac{\widetilde{K}(x,dy)}{\widetilde{K}(y,dx)} = \left( \frac{y}{x} \right)^{- \lambda -1} e^{- \frac{a^2}{2} \left( \frac{-1 + \sqrt{1+4xy}}{2x} + \frac{2x}{-1 + \sqrt{1+4xy}} - \frac{-1 + \sqrt{1+4xy}}{2y} - \frac{2y}{-1 + \sqrt{1+4xy}} \right)} \frac{dy}{dx}.$$
Multiplying the numerator and the denominator of the second and fourth term in the argument of the exponential by $1+ \sqrt{1+4xy}$, it leads to simplification:
$$\frac{\widetilde{K}(x,dy)}{\widetilde{K}(y,dx)} = \left( \frac{y}{x} \right)^{- \lambda -1} e^{- \frac{a^2}{2} \left( \frac{1}{y} - \frac{1}{x} \right)} \frac{dy}{dx} = \frac{k(y) dy}{k(x) dx} $$
where $k$ is defined by $k(x):= x^{- \lambda -1} e^{- \frac{a^2}{2x}} \mathds{1}_{\mathbb{R}_{+} ^*}(x)$. 
If $\lambda >0$, the integral of $k$ on $\mathbb{R}_{+}$ converges, so we normalize $k$ to obtain an invariant probability measure by: 
$$d \pi (x) = \frac{k(x)}{\int_{\mathbb{R}_{+}} k(y) dy} dx = \frac{a^{2 \lambda}}{2^{\lambda} \Gamma ( \lambda)} x^{- \lambda - 1} e^{- \frac{a^2}{2x}} \mathds{1}_{\mathbb{R}_{+} ^*}(x) dx.$$ 
\end{proof}

\noindent
From the expressions of $N_n ^{(\lambda)}$ in \eqref{ExpressionsdesN} and using Proposition \ref{PropBabillot} and Proposition \ref{Mesureinvariante} thanks to Remark \ref{ApplicationBabillot}, it follows:

\begin{theorem}[Dufresne identity]\label{Dufresne}
For $\lambda >0$, the law of the random variable 
$$N_{\infty} ^{(\lambda)} = \sum_{k=0}^{+ \infty} \gamma_k ^{-1} \left( \prod_{i=0}^{k-1} \gamma_i ^{-1} \right)^2 $$
is the inverse-gamma distribution $\pi$ defined by \eqref{MesInvariante}. 
\end{theorem}

\section{A discrete-time reconstruction theorem}
In this section, we give a discrete-time reconstruction theorem in the sense that we have formulas to recover $(X_n)_{n \in \mathbb{N}}$ from $(Z_n)_{n \in \mathbb{N}}$.

\begin{rem} \label{LGN2}
Let $\lambda >0$. From Remark \ref{LGN1}, we deduce that $X_n \overset{\text{a.s.}}{\to} + \infty$ when $n \to + \infty$. \\ Since $Z_n = X_n  N_{n} ^{(\lambda)}$, it follows from Proposition \ref{PropBabillot} that $Z_n \overset{\text{a.s.}}{\to} + \infty$ when $n \to + \infty$.
\end{rem}

\begin{proposition}
Let $\lambda > 0$. The process $(Z_n)_{n \in \mathbb{N}^*}$ is independent of the random variable $N_{\infty} ^{(\lambda)}$.
\end{proposition}
\begin{proof}
Let $k \geq 2$ fixed and let $h : \mathbb{R}^{k-1} \to \mathbb{R}$  and $g: \mathbb{R} \to \mathbb{R}$ be two bounded continuous functions. By Proposition \ref{PropBabillot} one has $\lim_{n \to + \infty} N_{n} ^{(\lambda)} = N_{\infty} ^{(\lambda)}$ a.s. it follows that: 
\begin{align*}
 \mathbb{E} \left( g(N_{\infty} ^{(\lambda)}) h(Z_2, \cdots, Z_k) \right) &= \lim_{n \to + \infty} \mathbb{E} \left( g(N_{n} ^{(\lambda)})  h(Z_2, \cdots, Z_k) \right)\\
&= \lim_{n \to + \infty} \mathbb{E} \left( \mathbb{E} \left( g(X_n ^{-1} Z_n)  h(Z_2, \cdots, Z_k) | Z_n, \cdots, Z_2 \right) \right) \\
&= \lim_{n \to + \infty} \mathbb{E} \left(  h(Z_2, \cdots, Z_k) \mathbb{E} \left( g(X_n ^{-1} Z_n) | Z_n, \cdots, Z_2 \right) \right).
\end{align*}
From Proposition \ref{Loiduvecteur} or directly from Proposition \ref{entrelacement}, we deduce:
\begin{align*}
\mathbb{E} \left( g(X_n ^{-1} Z_n) | Z_n, \cdots, Z_2 \right) &= \dfrac{1}{2 K_{\lambda} ( \frac{a^2}{Z_n})} \int_{0}^{+\infty} g(x^{-1} Z_n) x^{\lambda -1} e^{- \frac{a^2}{2 Z_n} \left( x + \frac{1}{x} \right)} dx \\
&= \dfrac{Z_n ^{\lambda}}{2 K_{\lambda} ( \frac{a^2}{Z_n})} \int_{0}^{+\infty} g(u) u^{-\lambda -1} e^{- \frac{a^2}{2 Z_n} \left( \frac{Z_n}{u} + \frac{u}{Z_n} \right)} du.
\end{align*}
where the second equality is obtained by a simple change of variable. The asymptotic when $z \to 0$ and $\lambda >0$ of Macdonald function is given by
$$K_{\lambda} (z) \sim \frac{1}{2} \Gamma(\lambda) \left( \frac{z}{2} \right)^{- \lambda},$$
see \cite{AbramoStegun} formula 9.6.9. which can be obtained from the asymptotic expansion of modified Bessel function of first kind and its relation with second kind. Thus, thanks to Remark \ref{LGN2} we obtain: 
$$ \lim_{n \to + \infty} \dfrac{Z_n ^{\lambda}}{2 K_{\lambda} ( \frac{a^2}{Z_n})} = \frac{a^{2 \lambda}}{2^\lambda \Gamma(\lambda)} \ \text{a.s.}$$
Finally, from Dufresne's identity (Theorem \ref{Dufresne}) it follows: 
$$\lim_{n \to + \infty} \mathbb{E} \left( g(X_n ^{-1} Z_n) | Z_n, \cdots, Z_2 \right) = \mathbb{E} \left( g(N_{\infty} ^{(\lambda)}) \right).$$
We get, 
$$\mathbb{E} \left( g(N_{\infty} ^{(\lambda)})  h(Z_2, \cdots, Z_k) \right) = \mathbb{E} \left( g(N_{\infty} ^{(\lambda)}) \right) \mathbb{E} \left(  h(Z_2, \cdots, Z_k) \right)$$ 
and the fact that the process $(Z_n)_{n \in \mathbb{N}^*}$ is independent of the random variable $N_{\infty} ^{(\lambda)}$.
 \end{proof}

\noindent
The following identity is easily established by
induction on $p$.

\begin{proposition}
For all $n, p \geq 0$,
\begin{equation}\label{Reconstruction1}
X_n = \frac{Z_n}{N_{n+p} ^{(\lambda)}} + Z_n \sum_{k=1}^{p} \frac{1}{Z_{n+k-1} Z_{n+k}}.
\end{equation}
\end{proposition}
\noindent
Factorizing by $\frac{Z_n}{N_{n+p} ^{(\lambda)}}$ and then taking the logarithm, we deduce from the previous equality that for all $n, p \geq 0$:
\begin{equation} \label{Inversehorizonfini}
\log ( X_n ) = \log \left( \frac{Z_n}{N_{n+p} ^{(\lambda)}} \right) + \log \left( 1 + N_{n+p} ^{(\lambda)} \sum_{k=1}^p \frac{1}{Z_{n+k-1} Z_{n+k}} \right).
\end{equation}

This last equality can be interpreted as a reconstruction theorem. Indeed, let us recall the inversion of Pitman transform in the simplest case. We refer to \cite{BBO} (Proposition 2.2) for a more general statement. Let us consider $\pi : [0 , T] \to \mathbb{R}$ a continuous function starting from $0$, i.e. $\pi(0)=0$. 
Let us define the Pitman transform $\mathcal{P}$ of $\pi$ by the formula, for all $t \in [0,T]$: 
$$\mathcal{P} \pi (t) := \pi (t) - 2 \inf_{0 \leq s \leq t} \pi (s).$$
Then, the continuous function $\pi$ can be recovered from the Pitman transform through the formula adding the information $\xi := - \inf_{0 \leq s \leq T} \pi(s)$, for all $t \in [0,T]$: 
\begin{equation}\label{InversePitman}
\pi (t) := \mathcal{P} \pi (t) - 2 \min \left( \xi, \inf_{t \leq s \leq T} \mathcal{P} \pi (s) \right).
\end{equation}
The formula \eqref{Inversehorizonfini} is a geometric analog of \eqref{InversePitman}.

\begin{theorem}
When $\lambda >0$ and for all $n \geq 0$,
\begin{equation} \label{MYconditionnel}
\log(X_n) = \log \left( \frac{Z_n}{N_{\infty} ^{(\lambda)}} \right) + \log \left( 1 + N_{\infty} ^{(\lambda)} \sum_{k=1}^{+ \infty} \frac{1}{Z_{n+k-1} Z_{n+k} } \right).
\end{equation}
When $\lambda \leq 0$ and for all $n \geq 0$,
\begin{equation}\label{Inversehorizoninfini}
\log( X_n )  = \log( Z_n ) + \log \left( \sum_{k=1}^{+ \infty} \frac{1}{Z_{n+k-1} Z_{n+k}} \right).
\end{equation}
\end{theorem}
\begin{proof}
When $\lambda >0$, letting  $p$ going to $+ \infty$ thanks to Proposition \ref{PropBabillot} and Lemma \ref{signeduloggamma}, we obtain \eqref{MYconditionnel} from \eqref{Inversehorizonfini}. In the case $\lambda \leq 0$ we get that $\lim_{p \to + \infty} N_{n+p} ^{(\lambda)} = + \infty$ a.s. so in formula \eqref{Reconstruction1} taking the logarithm we obtain \eqref{Inversehorizoninfini}
\end{proof}

The result \eqref{MYconditionnel} corresponds to the geometric analog, when $\pi$ has a positive drift, of the following formula with $\xi ' := - \inf_{0 \leq s} \pi(s)$, for all $t \geq 0$:
\begin{equation}\label{InversePitman2}
\pi (t) := \mathcal{P} \pi (t) - 2 \min \left( \xi', \ \inf_{t \leq s} \mathcal{P} \pi (s) \right).
\end{equation}

The result \eqref{Inversehorizoninfini} corresponds to the case $\inf_{0 \leq s} \pi (s) = - \infty$ when $\pi$ has nonpositive drift of the formula \eqref{InversePitman2} which becomes, for all $t \geq 0$:
\begin{equation}\label{InversePitman3}
\pi (t) := \mathcal{P} \pi (t) - 2 \inf_{t \leq s} \mathcal{P} \pi (s) .
\end{equation}

\section{Convergence towards the continuous-time Matsumoto--Yor process}

In this section we prove that for $T >0$, the sequence of processes $( b_{\lfloor nt \rfloor} ^{(\delta_n)}, 0 \leq t \leq T )_{n \in \mathbb{N}^*}$ converges weakly towards the continuous process \eqref{SolutionEDSREDA}. Here  $\delta_n = \frac{1}{n}$ and the parameters of the GIG law for the increments will be dependent of $n$. \\
Let $( \gamma_j ^{(a)} )_{j \geq 0}$ be a sequence of independent and identically distributed random variables with the law $GIG(\lambda, a, a)$ and $S_n ^{(a)}$ the random walk defined for all $n \geq 1$ by: 
$$S_n ^{(a)} := \sum_{j=0}^{n-1} \log \gamma_j ^{(a)}.$$
Then, $t \in [0, T] \mapsto S_{\lfloor n t \rfloor} ^{(a)}$  is a random variable with values in the Skorokhod space $\mathcal{D} \left( [0,T], \mathbb{R} \right)$. \\
To study the convergence of this random variable, we will use asymptotic formulas for the first moments of the law $\log(GIG)$.

\begin{proposition}
Let $m \in \mathbb{N}$. When $a \to + \infty$,
\begin{equation}\label{MomentsLogGIG}
\mathbb{E} \left( \log^m \gamma_0 ^{(a)} \right) \sim \left\{
    \begin{array}{ll}
       \frac{2^{\frac{m}{2}}}{a^m \sqrt{\pi}}  \Gamma \left( \frac{m+1}{2} \right) & \mbox{if} \ m \ \text{is even} \\     
       \frac{\lambda 2^{\frac{m+1}{2}}}{a^{m+1} \sqrt{\pi}}  \Gamma \left( \frac{m+2}{2} \right) & \mbox{if} \ m \ \text{is odd.} 
    \end{array}
\right.
\end{equation}
\end{proposition}
\begin{proof}
Using definition of the law $GIG(\lambda, a, a)$:
$$\mathbb{E} \left( \log ^m \gamma_0 ^{(a)} \right)= \frac{1}{2 K_{\lambda} (a^2)} \int_{0}^{+ \infty} \log ^m (x)  x^{\lambda - 1} e^{- \frac{a^2}{2} (x+ \frac{1}{x})} dx.$$
The same calculation as in the proof of Lemma \ref{signeduloggamma} gives: 
$$2 K_{\lambda} (a^2) \mathbb{E} \left( \log ^m \gamma_0 ^{(a)} \right) = \int_{1}^{+ \infty} \log ^m (x) \left( x^{\lambda} + (-1)^m x^{- \lambda} \right) e^{- \frac{a^2}{2} (x+ \frac{1}{x})} x^{-1} dx.$$
Putting $t= \frac{(x + x^{-1})}{2} - 1$ we obtain,
$$e^{a^2} K_{\lambda} (a^2) \mathbb{E} \left( \log^m \gamma_0 ^{(a)} \right) = \int_{0}^{+ \infty} f_m(t) e^{- a^2 t} dt $$
where,
$$f_m (t) := \left\{
    \begin{array}{ll}
      \frac{\argch^m (t+1) \cosh( \lambda \argch(t+1))}{\sinh( \argch(t+1))} & \mbox{if} \ m \ \text{is even} \\     
        \frac{\argch^m (t+1) \sinh( \lambda \argch(t+1))}{\sinh( \argch(t+1))} & \mbox{if} \ m \ \text{is odd.} 
    \end{array}
\right.$$
Now, it is easy to see that $f_m$ is exponentially bounded, i.e. there exists $b_m \in \mathbb{R}$ for $m \in \mathbb{N}$ such that:
$$ f_m (t) = O \left( e^{b_m t} \right) \ \text{when} \ t \to + \infty.$$ 
Therefore we have the following expansions when $t \to 0$, 
$$f_m (t) \sim \left\{
    \begin{array}{ll}
      (2t)^{\frac{m-1}{2}} & \mbox{if} \ m \ \text{is even} \\     
       \lambda (2t)^{\frac{m}{2}} & \mbox{if} \ m \ \text{is odd.} 
    \end{array}
\right.$$
Applying Lemma \ref{Watson}, we obtain the asymptotics formulas of moments when $a \to + \infty$.
\end{proof}

\begin{cor} \label{Donskerdrifté}
For all $T >0$, the sequence of processes 
$$
\left(S_{\lfloor nt \rfloor} ^{(\sqrt{n})} , 0 \leq t \leq T \right)_{n \in \mathbb{N}^*}$$
converges weakly in $\mathcal{D} \left( [0,T], \mathbb{R} \right)$ equipped with Skorokhod distance, towards the drifted Brownian motion $( B_t ^{(\lambda)} :=  B_t + \lambda t, t \geq 0 )$ where $(B_t, t \geq 0)$ is a standard Brownian motion.
\end{cor}

\begin{proof}
First of all, we prove the convergence in the sense of finite dimensional distributions. \\
First, we prove the convergence: 
$$S_{\lfloor nt \rfloor} ^{(\sqrt{n})} \overset{\text{law}}{\to} \mathbf{N}_{t} ^{(\lambda)} $$
where for $t>0$, $\mathbf{N}_{t} ^{(\lambda)} \sim \mathcal{N} (t \lambda, t) $ is a normal distribution with mean $t \lambda$ and variance $t$. \\
We consider Lindeberg's theorem (th. 27.3 in \cite{Billingsley}) on the triangular array with the random variables $V_{n,j}$ defined by: 
$$V_{n,j} := \frac{\log \gamma_j ^{(\sqrt{n})} - \mathbb{E} \left( \log \gamma_j ^{(\sqrt{n})} \right)}{\sqrt{\sum_{k=0}^{\lfloor nt \rfloor -1} \mathbb{V} \left( \log \gamma_j ^{(\sqrt{n})} \right)}} = \frac{\log \gamma_j ^{(\sqrt{n})} - \mathbb{E} \left( \log \gamma_0 ^{(\sqrt{n})} \right)}{\sqrt{ \lfloor nt \rfloor \mathbb{V} \left( \log \gamma_0 ^{(\sqrt{n})} \right)}}.$$
In order to apply the theorem, we establish the following Lyapounov's condition: 
$$\lim_{n \to + \infty} \sum_{j=0}^{ \lfloor nt \rfloor -1} \mathbb{E} \left| V_{n,j} \right| ^{4} = 0.$$
Using the inequality $|a+b|^4 \leq 2^4 ( |a|^4 + |b|^4)$, we get:
 $$\sum_{j=0}^{ \lfloor nt \rfloor -1} \mathbb{E} \left| V_{n,j} \right| ^{4} = \frac{\mathbb{E} \left|  \log \gamma_0 ^{(\sqrt{n})} - \mathbb{E} \left( \log \gamma_0 ^{(\sqrt{n})} \right)  \right|^4 }{ \lfloor nt \rfloor \mathbb{V} ^2 \left( \log \gamma_j ^{(\sqrt{n})} \right) } \leq \frac{2^5 \mathbb{E} \left( \log ^4 \gamma_0 ^{(\sqrt{n})} \right) }{\lfloor nt \rfloor \mathbb{V} ^2 \left( \log \gamma_j ^{(\sqrt{n})} \right)}.$$
Thanks to \eqref{MomentsLogGIG}, the right-hand side of the inequality is asymptotically $O \left( \frac{1}{n} \right)$ when $n \to + \infty$, so the Lyapounov condition holds. By Lindeberg's theorem, we obtain when $n \to + \infty$ that: 
$$ \sum_{j=0}^{\lfloor nt \rfloor - 1} V_{n,j} \overset{\text{law}}{\to} \mathbf{N}_1 ^{(0)}  \sim \mathcal{N}(0,1).$$
From \eqref{MomentsLogGIG} it follows when $n \to + \infty$:
$$S_{\lfloor nt \rfloor} ^{(\sqrt{n})} = \sqrt{ \lfloor nt \rfloor \mathbb{V} \left( \log \gamma_0 ^{(\sqrt{n})} \right)} \left(  \sum_{j=0}^{\lfloor nt \rfloor - 1} V_{n,j} \right) + \lfloor nt \rfloor \mathbb{E} \left( \log \gamma_j ^{(\sqrt{n})} \right) \overset{\text{law}}{\to} \mathbf{N}_{t} ^{(\lambda)}. $$
Using the fact that $S_{\lfloor nt_{k-1} \rfloor} ^{(\sqrt{n})}$ and $ S_{\lfloor nt_k \rfloor} ^{(\sqrt{n})} - S_{\lfloor nt_{k-1} \rfloor} ^{(\sqrt{n})}$ are independent, we obtain for all $k \in \mathbb{N}^*$ and for all $0 < t_1 < t_2 < \cdots < t_k <T$ that: 
$$\left( S_{\lfloor nt_1 \rfloor} ^{(\sqrt{n})} , S_{\lfloor nt_2 \rfloor} ^{(\sqrt{n})}, \dots, S_{\lfloor nt_k \rfloor} ^{(\sqrt{n})} \right) \overset{\text{law}}{\to} \left( \mathbf{N}_{t_1} ^{(\lambda)} , \mathbf{N}_{t_2} ^{(\lambda)} , \dots, \mathbf{N}_{t_k} ^{(\lambda)}  \right)$$
which proves the convergence in finite dimensional distributions. The tightness is obtained in the same way as in Donsker's theorem. We refer to \cite{Billingsley2} (th. 14.1) for the proof for the Brownian motion without drift.  
\end{proof}

\begin{theorem}
Let $t > 0$ fixed, we have the following convergence for the random walk \eqref{ExpressionRW} with $\delta_n  := \frac{1}{n}$ and $\gamma_i ^{(\sqrt{n})} \sim GIG( \lambda, \sqrt{n} , \sqrt{n} )$, when $n \to + \infty$:
$$b_{\lfloor nt \rfloor} ^{(\delta_n)} = \begin{pmatrix}
X_{\lfloor nt \rfloor} &0\\
Z_{\lfloor nt \rfloor} ^{(\delta_n)}  & X_{\lfloor nt \rfloor} ^{-1}
\end{pmatrix}  \overset{\text{law}}{\to} \begin{pmatrix}
e^{B_t ^{(\lambda)}} & 0  \\
e^{B_t ^{(\lambda)}} \int_{0}^{t} e^{- 2 B_s ^{(\lambda)}} ds & e^{- B_t ^{(\lambda)}}
\end{pmatrix}.$$
\end{theorem}
\begin{proof}
From Corollary \ref{Donskerdrifté} and by continuity of the exponential function, it is clear that $X_{\lfloor nt \rfloor} \overset{\text{law}}{\to} e^{B_t ^{(\lambda)}}$ and $X^{-1} _{\lfloor nt \rfloor} \overset{\text{law}}{\to} e^{-B_t ^{(\lambda)}}$ when $n \to + \infty$. Let $t>0$ fixed, we have
\begin{align*}
 e^{S_{\lfloor nt \rfloor} ^{(\sqrt{n})}} &\int_{0}^t e^{- 2 S_{\lfloor ns \rfloor} ^{(\sqrt{n})} - \log \gamma_{ \lfloor ns \rfloor} ^{(\sqrt{n})} } ds \\  
&= e^{S_{\lfloor nt \rfloor} ^{(\sqrt{n})}} \left( \sum_{k=0}^{\lfloor nt \rfloor - 1} \int_{\frac{k}{n}}^{\frac{k+1}{n}} e^{- 2 S_{\lfloor ns \rfloor} ^{(\sqrt{n})} - \log \gamma_{ \lfloor ns \rfloor} ^{(\sqrt{n})} } ds + \int_{\frac{\lfloor nt \rfloor}{n}}^{t} e^{- 2 S_{\lfloor ns \rfloor} ^{(\sqrt{n})} - \log \gamma_{ \lfloor ns \rfloor} ^{(\sqrt{n})} } ds \right)  \\
&= e^{S_{\lfloor nt \rfloor} ^{(\sqrt{n})}} \left(  \frac{1}{n} \sum_{k=0}^{\lfloor nt \rfloor - 1} e^{- 2 S_{k} ^{(\sqrt{n})} - \log \gamma_{ k} ^{(\sqrt{n})} }  + \left( t - \frac{\lfloor nt \rfloor}{n} \right)    e^{- 2 S_{\lfloor nt \rfloor} ^{(\sqrt{n})} - \log \gamma_{ \lfloor nt \rfloor} ^{(\sqrt{n})} } \right)  \\
&= Z_{\lfloor nt \rfloor} ^{(\delta_n)} + \left( t - \frac{\lfloor nt \rfloor}{n} \right) e^{- S_{\lfloor nt \rfloor} ^{(\sqrt{n})} - \log \gamma_{ \lfloor nt \rfloor} ^{(\sqrt{n})} }.
\end{align*}
Then, it follows the convergence in probability when $n \to + \infty$: 
\begin{equation}\label{CVenproba}
\left| e^{S_{\lfloor nt \rfloor} ^{(\sqrt{n})}} \int_{0}^t e^{ - 2 S_{\lfloor ns \rfloor} ^{(\sqrt{n})} - \log \gamma_{ \lfloor ns \rfloor} ^{(\sqrt{n})} } ds - Z_{\lfloor nt \rfloor} ^{(\delta_n)} \right|  \overset{\mathbb{P}}{\to} 0.
\end{equation}
Therefore we have
$$ e^{S_{\lfloor nt \rfloor} ^{(\sqrt{n})}} \int_{0}^t e^{ - 2 S_{\lfloor ns \rfloor} ^{(\sqrt{n})} - \log \gamma_{ \lfloor ns \rfloor} ^{(\sqrt{n})} } ds = \int_{0}^t e^{\sum_{k= \lfloor ns \rfloor  + 1}^{ \lfloor nt \rfloor - 1} \log \gamma_{k} ^{(\sqrt{n})}   -  S_{\lfloor ns \rfloor} ^{(\sqrt{n})}} ds . $$
Slightly adapting the proof of Corollary \ref{Donskerdrifté} we get that $ \left( \sum_{k= \lfloor ns \rfloor  + 1}^{ \lfloor nt \rfloor - 1} \log \gamma_{k} ^{(\sqrt{n})} , 0 \leq s \leq t \right)_{n \in \mathbb{N}^*}$ converges weakly towards $\left(  B_{t-s} ^{(\lambda)}, 0 \leq s \leq t \right)$. Hence, using the independence between $S_{\lfloor ns \rfloor} ^{(\sqrt{n})}$ and $\sum_{k= \lfloor ns \rfloor  + 1}^{ \lfloor nt \rfloor - 1} \log \gamma_{k} ^{(\sqrt{n})}$   
and by continuity of the map $I : \mathcal{D} \left( [0,T], \mathbb{R} \right) \to \mathbb{R}$ defined by the integral $I(f):=\int_{0}^t f(s) ds$ we obtain, when $n \to + \infty$: 
$$e^{S_{\lfloor nt \rfloor} ^{(\sqrt{n})}} \int_{0}^t e^{ - 2 S_{\lfloor ns \rfloor} ^{(\sqrt{n})} - \log \gamma_{ \lfloor ns \rfloor} ^{(\sqrt{n})} } ds 
\overset{\text{law}}{\to} \int_{0}^t e^{B_{t-s} ^{(\lambda)} - B_{s} ^{(\lambda)}} ds.$$
Finally, using the stationary increments property of Brownian motion and \eqref{CVenproba}, when $n \to + \infty$ we get that
$$Z_{\lfloor nt \rfloor} ^{(\delta_n)} \overset{\text{law}}{\to} e^{B_t ^{(\lambda)}} \int_{0}^{t} e^{- 2 B_s ^{(\lambda)}} ds. $$
\end{proof}

\begin{rem}
Let $G$ be the subgroup of lower triangular matrices in $SL_2$. We obtain with standard arguments for tightness that the sequence of processes $( b_{\lfloor nt \rfloor} ^{(\delta_n)}, 0 \leq t \leq T )_{n \in \mathbb{N}^*}$ converges weakly in $\mathcal{D}([0,T], G)$ towards the continuous process \eqref{SolutionEDSREDA}.
\end{rem}

\section{Convergence of the $N$-part towards stationary measure}\label{SectionCVNpart}
In this section, we give some elements of the theory of random walks on groups to state a result from \cite{Babillot} to prove Proposition \ref{PropBabillot}. \\

Let $G$ be the subgroup of $SL_2$ consisting of lower triangular matrices, defined by
$$G := \left\{ \begin{pmatrix}
x & 0 \\
z & x ^{-1}
\end{pmatrix}
\ | \ x \in \mathbb{R}^*, z \in \mathbb{R} \right\}.$$
Each element $h \in G$ can be decomposed as $h=n(h)a(h)$ or $h=a(h)\widetilde{n}(h)$ where for the matrix $h= \begin{pmatrix}
x & 0 \\
z & x ^{-1}
\end{pmatrix}$, we defined: 
$$n(h):=\begin{pmatrix}
1 & 0 \\
zx^{-1} & 1
\end{pmatrix}, \ \widetilde{n}(h):= \begin{pmatrix}
1 & 0 \\
zx & 1
\end{pmatrix} \ \text{and} \ a(h):= \begin{pmatrix}
x & 0 \\
0 & x ^{-1}
\end{pmatrix}.$$
Here, $G$ is a solvable Lie group such that $G=NA$ where $A$ is an abelian subgroup and $N$ is a nilpotent subgroup of $G$ with
$$
N := \left\{ \begin{pmatrix}
1 & 0 \\
z & 1
\end{pmatrix} \ | \ z \in \mathbb{R} \right\}
\ \text{and} \ 
A := \left\{ \begin{pmatrix}
x & 0 \\
0 & x ^{-1}
\end{pmatrix} \ | \ x \in \mathbb{R}^* \right\} .$$
Moreover, $G$ has a semidirect product structure defined on the set $N \times A $ with product law given by
$$h_1 h_2 = (n_1, a_2) (n_2, a_2) = (n_1 ( a_1 \odot n_2 ) , a_1 a_2) $$
where $n_i := n(h_i)$, $a_i := a(h_i)$ and $\odot$ is the group action of $A$ on $N$ by inner automorphisms in $G$ given, for 
$a \in A$ and $n \in N$, by 
$$a \odot n = ana^{-1},$$
this means, 
$$ \begin{pmatrix}
x & 0\\
0 & x ^{-1}
\end{pmatrix} 
\odot
\begin{pmatrix}
1 & 0 \\
u & 1
\end{pmatrix}
:=
\begin{pmatrix}
1 & 0 \\
x^{-2} u & 1
\end{pmatrix}.
$$
We obtain that $G$ acts on $N$ by
\begin{equation}\label{ActionGsurN}
\begin{pmatrix}
x & 0\\
z & x ^{-1}
\end{pmatrix} 
\cdot
\begin{pmatrix}
1 & 0 \\
u & 1
\end{pmatrix}
=
\begin{pmatrix}
1 & 0 \\
x^{-2} u + z x^{-1} & 1
\end{pmatrix}.
\end{equation} 

Let $\delta \in \mathbb{R}^*$ be a deterministic parameter. Recall that
\begin{equation}\label{increments}
g_n  := \begin{pmatrix}
\gamma_{n} & 0 \\
\delta & \gamma_{n}  ^{-1}
\end{pmatrix}
\end{equation}
are some i.i.d. random elements and let us denote by $\mu$ their law in the group $G$. Here, we denote by $g_n$ instead of $g_n ^{(\delta)}$ the increments of the random walk $(b_n ^{(\delta)})_{n \in \mathbb{N}}$ to simplify notation. \\
Then, we obtain by the recurrence relation (\ref{model}), for all $n \in \mathbb{N}^*$: 
$$ b_n ^{(\delta)} =  g_0  g_1  \cdots g_{n-1}.$$
With the previous definitions, the abelian part of the random walk $(b_n ^{(\delta)})_{n \in \mathbb{N}}$ is
$$a (b_n ^{(\delta)} ) = \exp \left( S_n \right) \ \text{with} \ S_n := \sum_{i=0}^{n-1} \log (a(g_i)).$$
Let $\mathcal{N}_n := n ( b_n ^{(\delta)} )$ be the $N$-part of the random walk in the $NA$ factorization of $G$. \\
We assume that $\mu$ has a finite first $\log$-moment and we denote by 
\begin{equation}\label{kappadef}
\kappa := \mathbb{E} \left( \log a(g_0) \right)
\end{equation} 
the mean of the increments of $S_n$. Then, the mean of the random walk $S_n$ is $n \kappa$. As we will see later, the convergence of the process $\left( \mathcal{N}_n \right)_{n \in \mathbb{N}}$ depends on the parameter $\kappa$.  \\
Let $\alpha$ be the unique simple root for the group $G$. That is to say, if $d := \text{diag}(a_1, a_2)$ we get 
$$\alpha(d) = a_2 - a_1.$$
Hence, we obtain
\begin{equation}\label{kappa}
\alpha( \kappa ) = -2 \mathbb{E} \left( \log \gamma_0 \right).
\end{equation}

\begin{definition} 
A probability measure $\mu$ on a group $G$ is said to be spread-out if there exists an integer $p$ such that $\mu ^{*p}$ is not singular with respect to a Haar measure on $G$. 
\end{definition}

\begin{rem} \label{RemarkHaar}
As it is shown in \cite{Revuz}, $\mu$ is spread-out on $G$ if and only if there exists an integer $q$ such that $\mu ^{*q}$ dominates a multiple of a Haar measure on a non-empty open subset of the group $G$. Moreover, here a right Haar measure for the group $G$ is given by
\begin{equation} \label{Haar}
d_H \begin{pmatrix}
x & 0 \\
z & x ^{-1}
\end{pmatrix} = \frac{dx \cdot dz}{x^2}.
\end{equation}
\end{rem}

\begin{definition}\label{convolutiondef}
Let $H$ be a group acting on a locally compact space $B$ by $H \times B \to B$, $(x,b) \mapsto x \cdot b$. 
The convolution of a probability measure $\mu$ on $H$ with a probability measure $\nu$ on $B$ is defined by 
$$\mu * \nu (\phi) := \int_{H \times B} \phi (x \cdot b) d \mu(x) d \nu (b).$$
A measure $\nu$ on $B$ will be called $\mu$-stationary, if it satisfies
$$\mu * \nu = \nu.$$
\end{definition}

Here, $G$ acts on $N$ by \eqref{ActionGsurN} and we consider the $\mu$-stationary measure on $N$.\\
From chapter 5 of the work of Babillot \cite{Babillot} (th. 5.11), we get the following proposition in the particular case of our random walk  $( b_n ^{(\delta)} )_{n \in \mathbb{N}}$.

\begin{proposition} \label{BabAN}
Assume that $\mu$ the probability measure defined in \eqref{increments} is spread-out on $G=NA$ with finite first $\log$-moment. Let $\kappa$ defined by \eqref{kappadef} such that $\alpha( \kappa ) < 0$ (contractive mean). Then, the $N$-component $\mathcal{N}_n$ of $b_n  ^{( \delta)} $ converges almost surely, and the law $\nu$ of $\mathcal{N}_{\infty} = \lim_{n \to + \infty} \mathcal{N}_n$ is the unique $\mu$-stationary measure on $N \simeq G / A$.
\end{proposition}

\begin{proof}[Proof of Proposition \ref{PropBabillot}]
Thanks to Proposition \ref{BabAN} we obtain Proposition \ref{PropBabillot}. Indeed, from \eqref{kappa}, if $\mathbb{E} \left( \log \gamma_0 \right) >0$, then $\alpha( \kappa ) < 0$. Moreover, $N$ is identified with the homogeneous space $G/A$ (which is locally compact) with origin $o=A$, so we consider Definition \ref{convolutiondef} with $H=G$ and $B=N$.
From the recurrence formula \eqref{RecurrenceFondamentale} we obtain for all $n \geq 2$: 
\begin{equation}\label{RecurrencePourNtildeGenerale}
\widetilde{N}_n (\gamma ^{-1}) = \gamma_{n-1} ^{-2} \widetilde{N}_{n-1}  (\gamma ^{-1}) + \gamma_{n-1} ^{-1}.
\end{equation}
Taking $z= \delta = 1$ in the formula \eqref{ActionGsurN}, we recover from \eqref{RecurrencePourNtildeGenerale} that the $\mu$-stationary measure on $N$ corresponds to the invariant probability measure of the process $( \widetilde{N}_n ( \gamma ^{-1} ) )_{n \in \mathbb{N}^*}$. 
\end{proof}

\begin{rem}\label{ApplicationBabillot}
We can use Proposition \ref{PropBabillot} when we consider GIG laws. Indeed, the sign of $\alpha( \kappa ) = -2 \mathbb{E} \left( \log \gamma_0 \right)$ is the same as $- \lambda$ due to Lemma \ref{signeduloggamma}. Moreover, the probability measure $\mu$ for the increments $g_i$ is spread-out as we show in next lemma.
\end{rem}

\begin{lemma}\label{spreadout}
Let $\mu$ be the probability measure \eqref{increments} on the group $G$ distributed such that $\gamma_i$ follows GIG law and $\delta$ follows Dirac distribution $\mathds{1}_{\delta = 1}$. Then, the measure $\mu$ is spread-out on $G$. 
\end{lemma}
\begin{proof}
Thanks to Remark \ref{RemarkHaar} it is sufficient to find a power $q$ such that $\mu ^{*q}$ dominates a multiple of Haar measure \eqref{Haar}. In fact, $q=2$ works since $\mu ^{*2}$ is given by the law of the random matrix:
$$\begin{pmatrix}
\gamma_{0} \gamma_{1} & 0 \\
\gamma_{0} ^{-1} + \gamma_1 & \gamma_{0} ^{-1} \gamma_{1} ^{-1}
\end{pmatrix}$$
and the pair $(\gamma_0 \gamma_1, \gamma_{0} ^{-1} + \gamma_1)$  has a density relative to Haar measure.
\end{proof}

\section{Appendix}

The following lemma is a slightly more general version of  Watson's lemma originally proved in 1918 in \cite{WatsonOriginal}. We refer to \cite{WatsonModerne} section 4.1 for a modern proof of the statement below.

\begin{lemma}[Watson] \label{Watson}
Let $f : \mathbb{R}_{+} \to \mathbb{C}$ be a locally integrable function satisfying two conditions: 
$(1) \  \text{There exists} \ b \in \mathbb{R} \ \text{such that} \ f(t) = O(e^{bt}) \ \text{when} \ t \to + \infty.\\ 
(2) \ \text{We have the following asymptotic expansion when}  \ t \to 0^+:$
$$f(t) \sim \sum_{n=0}^{+ \infty} c_n t^{a_n}$$
where $(\Re(a_n))_{n \in \mathbb{N}}$ increases monotonically to $+ \infty$ and such that $\Re(a_0)  > - 1 $. \\
Then, we have the asymptotic expansion for the Laplace transform: 
$$I(x) := \int_{0}^{+ \infty} f(t) e^{-xt} dt \sim \sum_{n=0}^{+ \infty} c_n \frac{\Gamma (a_n +1)}{x^{a_n + 1}} \ \text{when} \ x \to + \infty.$$
\end{lemma}

\begin{lemma} \label{EqFon}
The only continuous solutions on $\mathbb{R}_{+}$ of the functional equation:
$$( 4 x^2 - 1) \varphi(2x) + (1-x^2) \varphi(x)= 3x^2 \varphi(2)$$
which also satisfy the equation $\varphi(x)=\varphi \left( \frac{1}{x} \right)$ are constants.
\end{lemma}
\begin{proof}
Let us set $\beta(x) := \frac{x^2 - 4}{4(x^2 -1)}$ and $\alpha(x) := \frac{3x^2}{4(x^2 - 1)}$. Replacing $x$ by $\frac{x}{2}$ in the functional equation, we obtain for all $x \in \mathbb{R}_{+} \backslash \left\{ 1 \right\}$
$$ \varphi(x)= \beta(x) \varphi \left( \frac{x}{2} \right) + \alpha(x) \varphi(2).$$
Iterating this equality, we get:
\begin{equation} \label{EFiterated}
\varphi(x) = \left( \prod_{k=0}^n \beta \left( \frac{x}{2^k} \right) \right) \varphi \left( \frac{x}{2^{n+1}} \right) + \varphi(2) \sum_{k=0}^n \alpha \left( \frac{x}{2^k} \right) \prod_{j=0}^{k-1} \beta \left( \frac{x}{2^j} \right) .
\end{equation}
Furthermore, we get for all $k \geq 0$: $\beta \left( \frac{x}{2^k} \right) = \frac{(x-2^{k+1})(x+2^{k+1})}{4(x-2^k)(x+2^k)}$. Hence, we obtain
$$\forall n \geq 0: \  \prod_{k=0}^n \beta \left( \frac{x}{2^k} \right) = \frac{x^2 - 4^{n+1}}{4^{n+1} (x^2 - 1)}.$$
Thus, on the one hand,
$$\lim_{n \to + \infty} \prod_{k=0}^n \beta \left( \frac{x}{2^k} \right) = - \frac{1}{x^2 - 1}$$
and on the other hand,
$$\lim_{ n \to + \infty} \sum_{k=0}^n \alpha \left( \frac{x}{2^k} \right) \prod_{j=0}^{k-1} \beta \left( \frac{x}{2^j} \right) = \lim_{ n \to + \infty} \frac{3x^2}{4(x^2 -1)} \sum_{k=0}^n \left( \frac{1}{4} \right)^k = \frac{x^2}{x^2 - 1}.$$
Thus, letting $n$ going to $+\infty$ in \eqref{EFiterated}, we obtain by continuity
\begin{equation} \label{LastFE}
\forall x \in \mathbb{R}_{+} \backslash \left\{ 1 \right\} , \ \varphi(x) = - \frac{1}{x^2 - 1} \varphi(0) + \frac{x^2}{x^2 - 1} \varphi(2).
\end{equation}
This last equality combined to the equality $\varphi(x)=\varphi \left( \frac{1}{x} \right)$ gives $\varphi(0) = \varphi(2)$. Substituting this into \eqref{LastFE} we get: $\forall x \in \mathbb{R}_{+} \backslash \left\{ 1 \right\}, \ \varphi(x) = \varphi(2)$. By continuity this last equality holds for all $x \in \mathbb{R}_{+}$.
\end{proof}


\begin{thebibliography}{}
\bibitem[1]{AbramoStegun}
\textbf{Milton Abramowitz and Irene A. Stegun}. \textit{Handbook of Mathematical Functions with Formulas, Graphs, and Mathematical Tables}. New York: Dover, 1970.  

\bibitem[2]{AristaOconnell}
\textbf{Jonas Arista, Elia Bisi and Neil O'Connell}. \textit{Matsumoto–-Yor and Dufresne type theorems for a random walk on positive definite matrices}. Ann. Inst. H. Poincaré Probab. Statist. 60(2): 923-945, (2024).

\bibitem[3]{Babillot}
\textbf{Martine Babillot}. \textit{An introduction to Poisson boundaries of Lie groups}. Probability measures on groups: recent directions and trends, 1-90, Tata Inst. Fund. Res., Mumbai, 2006. 

\bibitem[4]{Bertoin}
\textbf{Jean Bertoin}. \textit{An extension of Pitman’s theorem for spectrally positive Lévy processes}. Ann.
Probab., 20(3), (1992), 1464–1483.

\bibitem[5]{Biane1}
\textbf{Philippe Biane}. \textit{Marches de Bernoulli quantiques}. Séminaire de probabilités (Strasbourg), tome
24 (1990), 329–344.

\bibitem[6]{BianeIntertwining}
\textbf{Philippe Biane}. \textit{Intertwining of Markov semi-groups, some examples}. Séminaire de probabilités (Strasbourg), tome 29 (1995), 30-36

\bibitem[7]{Biane2}
\textbf{Philippe Biane}. \textit{Le théorème de Pitman, le groupe quantique $SU_q(2)$, et une question de P. A.
Meyer}. Séminaire de Probabilités XXXIX, Lecture Notes in Math., 1874, Springer-Verlag (2006), 61–75.

\bibitem[8]{BBO}
\textbf{Philippe Biane, Philippe Bougerol and Neil O’Connell}. \textit{Littelmann paths and Brownian paths}. Duke Math. J. 130 (2005), no. 1, 127-167.

\bibitem[9]{BBO2}
\textbf{Philippe Biane, Philippe Bougerol and Neil O’Connell}. \textit{Continuous crystal and Duistermaat-Heckman measure for Coxeter groups}. Adv. Maths. 221 (2009) 1522-1583.

\bibitem[10]{Billingsley}
\textbf{Patrick Billingsley}. \textit{Probability and Measure}. Wiley Series in Probability and Mathematical Statistics, Third Edition, 1995.

\bibitem[11]{Billingsley2}
\textbf{Patrick Billingsley}. \textit{Convergence of Probability Measures}. Wiley Series in Probability and Mathematical Statistics, Second Edition, 1999.

\bibitem[12]{WatsonModerne}
\textbf{Norman Bleistein and Richard A. Handelsman}. \textit{Asymptotic expansions of integrals}. Second edition. Dover Publications, Inc., New York, 1986.

\bibitem[13]{Bou-Def}
\textbf{Philippe Bougerol and Manon Defosseux}. \textit{Pitman transforms and Brownian motion in the interval viewed as an affine alcove}. Annales Scientifiques de l’École Normale Supérieure, (2022), 55 (2), 429–472.

\bibitem[14]{CPY}
\textbf{Philippe Carmona, Frédérique Petit and Marc Yor}. \textit{Beta-gamma random variables
and intertwining relations between certain Markov processes}. Revista Matemática Iberoamericana 14.2 (1998), 311-367.

\bibitem[15]{ChamayouLetac}
\textbf{Jean-François Chamayou and Gérard Letac}. \textit{Additive properties of the Dufresne laws and their multivariate extension}. J. Theoret. Probab. 12, pp. 1045–1066 (1999).

\bibitem[16]{Reda}
\textbf{Reda Chhaibi}. \textit{Modèle de Littelmann pour cristaux géométriques, fonctions de Whittaker sur des
groupes de Lie et mouvement brownien}. Probabilités [math.PR]. Université Pierre et Marie Curie -
Paris VI, 2013.

\bibitem[17]{Reda1}
\textbf{Reda Chhaibi}. \textit{Littelmann path model for geometric crystals}. arXiv:1405.6437v2 [math.RT] 

\bibitem[18]{Reda2}
\textbf{Reda Chhaibi}. \textit{Whittaker processes and Landau-Ginzburg potentials for flag manifolds}. 
arXiv:1504.07321 [math.PR]

\bibitem[19]{Reda3}
\textbf{Reda Chhaibi}. \textit{Beta-gamma algebra identities and Lie-theoretic exponential functionals of Brownian motion}. Electron. J. Probab. 20 1 - 20, 2015. https://doi.org/10.1214/EJP.v20-3666

\bibitem[20]{Corwin}
\textbf{Ivan Corwin, Neil O’Connell, Timo Seppäläinen and Nikos Zygouras}. \textit{Tropical combinatorics and
Whittaker functions}. Duke Math. J., 163(3):513–563, 2014.


\bibitem[21]{Def-Her}
\textbf{Manon Defosseux and Charlie Hérent}. \textit{A converse to Pitman's theorem for a space-time Brownian motion in a type $A_1^1$ Weyl chamber}. arXiv:2211.03439v2 [math.PR]

\bibitem[22]{Dufresne}
\textbf{Daniel Dufresne}. \textit{The distribution of a perpetuity, with application to risk theory and pension funding}. Scand. Actuarial. J. (1990), 39–79.

\bibitem[23]{Halphen}
\textbf{Étienne Halphen}. \textit{Sur un nouveau type de courbe de fréquence}. Comptes Rendus de l’Académie des Sciences 213, 633–635, 1941. Published under the name of “Dugué” due to war constraints.

\bibitem[24]{Good}
\textbf{Irving J. Good}. \textit{The population frequencies of species and the estimation of population parameters}. Biometrika, 40 (1953), 237–264.

\bibitem[25]{KoudouLey}
\textbf{Efoevi Koudou and Christophe Ley}. \textit{Characterizations of GIG laws: A survey}. Probability Surveys, Vol. 11 pp.161-176, 2014. 

\bibitem[26]{Koudou}
\textbf{Efoevi Koudou and Pierre Vallois}. \textit{Independence properties of the Matsumoto–Yor type}. Bernoulli 18 (1) (2012)  119–136. 

\bibitem[27]{Letac}
\textbf{Gérard Letac and Vanamamalai Seshadri}. \textit{A characterization of the generalized inverse Gaussian distribution by continued fractions}. Z. Wahrsch. Verw. Gebiete 62 (1983), no.4, 485–489.

\bibitem[28]{Letac2}
\textbf{Gérard Letac and Jacek Wesołowski}. \textit{An independence property for the product of GIG and gamma laws}.  Ann. Probab. 28 (2000), no.3, 1371–1383.

\bibitem[29]{MY0}
\textbf{Hiroyuki Matsumoto and Marc Yor}. \textit{A version of Pitman's $2M-X$ theorem for geometric Brownian motions, Une version du théorème de Pitman pour mouvements browniens géométriques}.  C. R. Acad. Sci. Paris, t. 328, Série I, p. 1067-1074, 1999.

\bibitem[30]{MY1}
\textbf{Hiroyuki Matsumoto and Marc Yor}. \textit{An analogue of Pitman's $2M-X$ theorem for exponential Wiener functionals. I. A time-inversion approach}.  Nagoya Math. J. 159 125 - 166, 2000.

\bibitem[31]{MY2}
\textbf{Hiroyuki Matsumoto and Marc Yor}. \textit{An analogue of Pitman's $2M-X$ theorem for exponential Wiener functionals. II. The role of the generalized inverse Gaussian laws}.  Nagoya Math. J. 162 65 - 86, 2001.

\bibitem[32]{MY3}
\textbf{Hiroyuki Matsumoto and Marc Yor}. \textit{A relationship between Brownian motions with opposite drifts via certain enlargements of the Brownian filtration}. Osaka J. Math. 38(2): 383-398, 2001.

\bibitem[33]{Noumi-Yamada}
\textbf{Masatoshi Noumi and Yasuhiko Yamada}. \textit{Tropical Robinson-Schensted-Knuth correspondence and birational Weyl group actions}. Adv. Stud. Pure Math. (2004) 371-442

\bibitem[34]{OconnellFirst}
\textbf{Neil O'Connell}. \textit{Directed polymers and the quantum Toda lattice}. Ann. Probab. 40(2), 437-458 (2012)

\bibitem[35]{Oconnell}
\textbf{Neil O'Connell, Timo Seppäläinen and Nikos Zygouras}. \textit{Geometric RSK correspondence, Whittaker functions
and symmetrized random polymers}. Invent. Math. 197(2), 361–416 (2014)

\bibitem[36]{Pitman}
\textbf{James W. Pitman}. \textit{One-dimensional Brownian motion and the three-dimensional Bessel process}. Adv. Appl. Probab. 7 (1975) 511–526.

\bibitem[37]{Revuz}
\textbf{Daniel Revuz}. \textit{Markov chains}. North-Holland Mathematical Library, 11, North-Holland, 1984

\bibitem[38]{RiderValko}
\textbf{Brian Rider and Benedek Valkó}. \textit{Matrix Dufresne identities}. Int. Math. Res. Not., 2016(1), pp. 174–218 (2016).

\bibitem[39]{RogersPitman}
\textbf{L. C. G. Rogers and James W. Pitman }. \textit{Markov Functions}. Ann. Probab. 9 (4) 573 - 582 (1981).

\bibitem[40]{Vallois}
\textbf{Pierre Vallois}. \textit{La loi gaussienne inverse généralisée comme premier ou dernier temps de passage de diffusions}. Bull. Sc. math., 2e série, 115 (1991), 301–368.

\bibitem[41]{WatsonOriginal}
\textbf{George N. Watson}. \textit{The Harmonic Functions Associated with the Parabolic Cylinder}. Proc. London Math. Soc. (2) 17 (1918), 116–148.
\end{thebibliography}
\end{document}